\documentclass[fontsize=10pt,abstract=on]{scrartcl}
\usepackage{mathptmx}
\usepackage{amsthm}
\usepackage{amssymb}
\usepackage{amsmath}
\usepackage{enumerate}
\usepackage[pdfpagelabels,colorlinks,linkcolor=blue]{hyperref}

\newtheorem{theorem}{Theorem}
\newtheorem{lemma}[theorem]{Lemma}

\theoremstyle{definition}

\newtheorem{remark}[theorem]{Remark}
\newtheorem{condition}[theorem]{Condition}

\allowdisplaybreaks

\let\originalleft\left
\let\originalright\right
\renewcommand{\left}{\mathopen{}\mathclose\bgroup\originalleft}
\renewcommand{\right}{\aftergroup\egroup\originalright}
\DeclareMathOperator{\supp}{supp}
\DeclareMathOperator{\divg}{div}
\DeclareMathOperator{\curl}{curl}
\newcommand{\R}{\mathbb{R}}
\newcommand{\inte}{\mathrm{int}}
\newcommand{\ext}{\mathrm{ext}}
\newcommand{\numb}{\stepcounter{equation}\tag{\theequation}}
\newcommand{\phan}{\phantom{=\ }}

\begin{document}

\title{Optimal Control of the Two-Dimensional Vlasov-Maxwell System}
\author{J\"org Weber\\ \textit{University of Bayreuth, 95440 Bayreuth, Bavaria, Germany}\\ \texttt{Joerg.Weber@uni-bayreuth.de}}
\date{}

\maketitle
\begin{abstract}
	The time evolution of a collisionless plasma is modeled by the Vlasov-Maxwell system which couples the Vlasov equation (the transport equation) with the Max\-well equations of electrodynamics. We only consider a 'two-dimensional' version of the problem since existence of global, classical solutions of the full three-dimensional problem is not known. We add external currents to the system, in applications generated by coils, to control the plasma in a proper way. After considering global existence of solutions to this system, differentiability of the control-to-state operator is proved. In applications, on the one hand, we want the shape of the plasma to be close to some desired shape. On the other hand, a cost term penalizing the external currents shall be as small as possible. These two aims lead to minimizing some objective function. We restrict ourselves to only such control currents that are realizable in applications. After that, we prove existence of a minimizer and deduce first order optimality conditions and the adjoint equation.
\end{abstract}

\section{Introduction}
\subsection{The system}
The time evolution of a collisionless plasma is modeled by the Vlasov-Maxwell system. Collisions among the plasma particles can be neglected if the plasma is sufficiently rarefied or hot. The particles only interact through electromagnetic fields created collectively. We only consider plasmas consisting of just one particle species, for example, electrons. This work can immediately be adapted to the case of several particle species. For the sake of simplicity, we choose units such that physical constants like the speed of light, the charge and rest mass of an individual particle are normalized to unity. Allowing the particles to move at relativistic speeds, the three-dimensional Vlasov-Maxwell system is given by
\begin{subequations}
\begin{align}
\partial_t f+\widehat{p}\cdot\partial_x f+\left(E+\widehat{p}\times B\right)\cdot\partial_p f &= 0,\label{vl}\\
\partial_t E-\curl_xB &= -j_f,\label{mw1}\\
\partial_t B+\curl_xE &= 0,\\
\divg_xE &= \rho\label{mw3},\\
\divg_xB &= 0\label{mw2},\\
\rho_f &= 4\pi\int{f\,dp},\label{rho}\\
j_f &= 4\pi\int{\widehat{p}f\,dp}\label{j}.
\end{align}
\end{subequations}
Here, the Vlasov equation is \eqref{vl} and the Maxwell equations of electrodynamics are \eqref{mw1} to \eqref{mw2}. Vlasov and Maxwell equations are coupled via \eqref{rho} and \eqref{j} rendering the whole system nonlinear due to the product term $\left(E+\widehat{p}\times B\right)\cdot\partial_p f$. In particular, $f=f\left(t,x,p\right)$ denotes the density of the particles on phase space, and $E=E\left(t,x\right)$, $B=B\left(t,x\right)$ are the electromagnetic fields, whereby $t\in\R$, $x$, and $p\in\R^3$ stand for time, position in space, and momentum. The abbreviation $\widehat{p}=\frac{p}{\sqrt{1+\left| p\right|^2}}$ denotes the velocity of a particle with momentum $p$. Furthermore, some moments of $f$ appear as source terms in the Maxwell equations, that is to say $j_f$ and $\rho_f$ which equal the current and charge density up to the constant $4\pi$.

However, we have not readily explained the source term $\rho$ in \eqref{mw3}. If we would demand $\divg_xE=\rho_f$ this would lead to a seeming contradiction: Formally integrating this equation with respect to $x$ (and assuming $E\rightarrow 0$ rapidly enough at $\infty$) leads to $\int{\rho_f\,dx=0}$ and hence $f=0$ by $\mathring{f}\geq 0$. This problem is caused by our simplifying restriction to one species of particles and is resolved by adding some terms to $\rho_f$, for example a neutralizing background density, so that we have a total charge density $\rho$ with vanishing space integral.

Considering the Cauchy problem for the above system, we moreover demand
\begin{align*}
f\left(0,x,p\right)=\mathring{f}\left(x,p\right),E\left(0,x\right)=\mathring{E}\left(x\right),B\left(0,x\right)=\mathring{B}\left(x\right),
\end{align*}
where $\mathring{f}\geq 0$, $\mathring{E}$, and $\mathring{B}$ are some given initial data.

Unfortunately, existence of global (i.e., global in time), classical (i.e., continuously differentiable) solutions for general (smooth) data is an open problem in the three-dimensional setting. It is only known that global weak solutions can be obtained. This was proved by R.J. Di Perna and P.L. Lions \cite{dipl}. For a detailed insight concerning this matter we recommend the review article \cite{gws} by G. Rein. As for global existence of classical solutions, the strategy was to first consider lower dimensional settings. R. Glassey and J. Schaeffer proved global existence of classical solutions in the one and one-half \cite{oah}, the two \cite{rvm1,rvm2}, and the two and one-half dimensional setting \cite{tah}.

Since it is convenient to have global existence of classical solutions on hand, we consider a two-dimensional version of the problem in this work. Notice that \textit{mutatis mutandis} all results and techniques can be applied to the full three-dimensional setting once global existence of classical solutions has been proved. The restriction to 'two-dimensionality' is to be understood in the following sense: All functions shall be independent of the third variables $x_3$ and $p_3$. This new model describes a plasma where the particles only move in the $\left(x_1,x_2\right)$-plane, but the plasma extends in the $x_3$-direction infinitely. To ensure that these properties are preserved in time, we have to demand that the electric field lies in the plane and that the magnetic field is perpendicular to the plane so that $E=\left(E_1\left(t,x\right),E_2\left(t,x\right),0\right)$ and $B=\left(0,0,B\left(t,x\right)\right)$. Here and in the following, let $x=\left(x_1,x_2\right)$ and $p=\left(p_1,p_2\right)$ be two-dimensional variables. Note that hence the magnetic field is always divergence free with respect to $x$, so that \eqref{mw2} is always satisfied and will no longer be mentioned. The two-dimensional Vlasov-Maxwell system reads
\begin{align*}
\partial_t f+\widehat{p}\cdot\partial_x f+\left(E+\left(\widehat{p}_2,-\widehat{p}_1\right)B\right)\cdot\partial_p f &= 0,\\
\partial_t E_1-\partial_{x_2}B &= -j_{f,1},\\
\partial_t E_2+\partial_{x_1}B &= -j_{f,2},\\
\partial_t B+\partial_{x_1}E_2-\partial_{x_2}E_1 &= 0,\\
\divg_xE &= \rho,\\
\left.\left(f,E,B\right)\right|_{t=0} &= \left(\mathring{f},\mathring{E},\mathring{B}\right).
\end{align*}
The goal is to control the plasma in a proper way. Thereto we add external currents $U$ to the system, in applications generated by electric coils. These currents, like the electric field and the current density of the plasma particles, have to lie in the plane and have to be independent of the third space coordinate. Of course, there will be an external charge density $\rho_\ext$ corresponding to the external current. It is natural to assume local conservation of the external charge, i.e.
\begin{align*}
\partial_t\rho_\ext+\divg_xU=0.
\end{align*}
Hence we can eliminate $\rho_\ext$ via
\begin{align*}
\rho_\ext=\mathring{\rho}_\ext-\int_0^t{\divg_xU\,d\tau}.
\end{align*}
The initial value $\mathring{\rho}_\ext$ will be added to the background density. This total background density will be neglected throughout this work.

Also, for simplicity, we do not consider material parameters, for example for modeling supraconductors in a fusion reactor, that is to say permittivity and permeability, which would appear in the Maxwell equations.

In the following, we consider the controlled relativistic Vlasov-Maxwell system
\begin{align}\tag{CVM}\label{CVM}\left.\begin{aligned}
\partial_t f+\widehat{p}\cdot\partial_x f+\left(E-\widehat{p}^\bot B\right)\cdot\partial_p f &= 0,\\
\partial_t E_1-\partial_{x_2}B &= -j_{f,1}-U_1,\\
\partial_t E_2+\partial_{x_1}B &= -j_{f,2}-U_2,\\
\partial_t B+\partial_{x_1}E_2-\partial_{x_2}E_1 &= 0,\\
\divg_xE &= \rho_f-\int_0^t{\divg_xU\,d\tau},\\
\left.\left(f,E,B\right)\right|_{t=0} &= \left(\mathring{f},\mathring{E},\mathring{B}\right)
\end{aligned}\right\}\end{align}
on a finite time interval $\left[0,T\right]$ with given $T>0$; here we introduced the abbreviation $a^\bot=\left(-a_2,a_1\right)$ for $a\in\R^2$.

It is well known that $L^q$-norms (with respect to $\left(x,p\right)$, $1\leq q\leq\infty$) of $f$ are preserved in time by $f$ solving the Vlasov equation since the vector field $\left(\widehat{p},E-\widehat{p}^\bot B\right)$ is divergence free in $\left(x,p\right)$. Therefore, especially, the $L^1$-norm (with respect to $x$) of the charge density $\rho_f$ is constant in time.

The outline of our work is the following: In the first part, we have to prove unique solvability of \eqref{CVM}. Of course, some regularity assumptions on the external current and the initial data have to be made in order to prove existence of classical solutions. In the second part, we consider an optimal control problem. On the one hand, we want the shape of the plasma to be close to some desired shape. On the other hand, the external currents shall be as small as possible. These two aims lead to minimizing some objective function. To analyze the optimal control problem, it is convenient to show differentiability of the control-to-state operator first. After that, we prove existence of a minimizer and deduce first order optimality conditions and the adjoint equation.

The steps mentioned above were carried out by P. Knopf \cite{knopf} and, only considering realizable control fields, by Knopf and the author \cite{kw} for the three-dimensional Vlasov-Poisson system with an external magnetic field. The consideration of the latter setting has the advantage of being able to work in three dimensions, but has the disadvantage of only imposing Poisson's equation, that is, Maxwell's equations with an internal magnetic field sufficiently small to be neglected, for the electromagnetic fields, which make things easier due to the elliptic nature of Poisson's equation in contrast to the hyperbolic nature of the (time evolutionary) Maxwell equations.

Also other approaches for controlling a Vlasov-Maxwell plasma have been considered in the literature, but they are different in nature compared to our approach. We refer to \cite{GH15,NNS15} and the references therein.

\subsection{Some notation and simple computations}
We denote by $B_r\left(x\right)$ the open ball with radius $r>0$ and center $x\in X$ where $X$ is a normed space. Furthermore, we abbreviate $B_r:=B_r\left(0\right)$. For a function $g\colon\left[0,T\right]\times\R^j\rightarrow\R^k$ we abbreviate $g\left(t\right):=g\left(t,\cdot\right)\colon\R^j\rightarrow\R^k$ for $0\leq t\leq T$. Also, we write $\supp g$ for the support of $g$, and $\supp_xg$ (and likewise $\supp_pg$) for the support of a function $g=g\left(t,x,p\right)$ with respect to $x$, that is, the closure of the set of all $x$ such that there are $t$ and $p$ with $g\left(t,x,p\right)\neq 0$. Sometimes, denoting certain function spaces, we omit the set where these functions are defined. Which set is meant should be obvious, in fact the largest possible set like $\left[0,T\right]\times\R^j$ (including time) or $\R^j$ (not including time). Moreover, $C_b^k$ denotes the space of $k$-times continuously differentiable functions (on a given set) such that all derivatives up to order $k$ are bounded. The index $c$, as in $C_c^k$, indicates that such functions are compactly supported. Furthermore, $X\hookrightarrow Y$ means that $X$ is continuously embedded in $Y$. Finally, we use the abbreviations
\begin{align*}
\xi=\frac{y-x}{t-\tau},\,es=\frac{-2\left(\xi+\widehat{p}\right)}{1+\widehat{p}\cdot\xi},\,bs=\frac{-2\xi\cdot\widehat{p}^\bot}{1+\widehat{p}\cdot\xi},\\
et=\frac{-2\left(1-\left|\widehat{p}\right|^2\right)\left(\xi+\widehat{p}\right)}{\left(1+\widehat{p}\cdot\xi\right)^2},\,bt=\frac{-2\left(1-\left|\widehat{p}\right|^2\right)\xi\cdot\widehat{p}^\bot}{\left(1+\widehat{p}\cdot\xi\right)^2},
\end{align*}
where $t,\tau\in\left[0,T\right]$, $x,y,p\in\R^2$.

We state some fundamental properties which will be used several times:
\begin{remark}\label{fund}
\begin{enumerate}[i)]
	\item For $\left| p\right|\leq r$ and $\left|\xi\right|\leq 1$ we can estimate
	\begin{align*}
	\left|\partial_p\left(bs\right)\right|,\left|\partial_p\left(es\right)\right|,\left|\partial_p\partial_\xi\left(bs\right)\right|,\left|\partial_p\partial_\xi\left(es\right)\right|,\left| bt\right|,\left| et\right|,\left|\partial_{\left(\xi,p\right)}\left(bt\right)\right|,\left|\partial_{\left(\xi,p\right)}\left(et\right)\right|
	\end{align*}
	by a constant $C\left(r\right)>0$ only depending on $r$, since
	\begin{align*}
	\left|1+\widehat{p}\cdot\xi\right|\geq1-\left|\widehat{p}\right|\left|\xi\right|\geq1-\frac{r}{\sqrt{1+r^2}}>0.
	\end{align*}
	\item We compute
	\begin{align*}
	\int_{\left| x-y\right|<t-\tau}{\frac{dy}{\sqrt{\left(t-\tau\right)^2-\left| x-y\right|^2}}}=2\pi\int_0^{t-\tau}{s\left(\left(t-\tau\right)^2-s^2\right)^{-\frac{1}{2}}ds}=2\pi\left(t-\tau\right)
	\end{align*}
	and
	\begin{align*}
	&\int_0^t{\int_{\left| x-y\right|<t-\tau}{\frac{dyd\tau}{\left(t-\tau\right)^{l+1}\sqrt{1-\left|\xi\right|^2}}}}=\int_0^t{\int_{\left| x-y\right|<t-\tau}{\frac{dyd\tau}{\left(t-\tau\right)^l\sqrt{\left(t-\tau\right)^2-\left| x-y\right|^2}}}}\\
	&=2\pi\int_0^t\left(t-\tau\right)^{-l+1} d\tau\leq\frac{2\pi}{2-l}T^{2-l}=C\left(T,l\right)<\infty
	\end{align*}
	for $l<2$.
\end{enumerate}
\end{remark}
\subsection{Maxwell equations}
We will have to consider first order and second order Maxwell equations. It is well known that they are equivalent and that the divergence equations propagate in time if local conservation of charge holds, i.e. 
\begin{align}\tag{LC}\label{LC}\begin{aligned}
\partial_t\rho+\divg_xj=0.
\end{aligned}\end{align}
In our two-dimensional setting with fields $\left(E_1,E_2,0\right)$ and $\left(0,0,B\right)$ we conclude:
\begin{lemma}\label{equi}
Let $\mathring{E}$ and $\mathring{B}$ be of class $C^2$ and $E$, $B\in C^2$, and $\rho$, $j\in C^1$. If the conditions
\begin{align}\tag{CC}\label{CC}\begin{aligned}
\divg\mathring{E}=\rho\left(0\right)
\end{aligned}\end{align}
and
\begin{align}\tag{LC}\begin{aligned}
\partial_t\rho+\divg_xj=0
\end{aligned}\end{align}
are satisfied, then the systems of first order Maxwell equations
\begin{align}\tag{1stME}\left.\begin{aligned}
\partial_t E_1-\partial_{x_2}B &= -j_1,\\
\partial_t E_2+\partial_{x_1}B &= -j_2,\\
\partial_t B+\partial_{x_1}E_2-\partial_{x_2}E_1 &= 0,\\
\left(E,B\right)\left(0\right) &= \left(\mathring{E},\mathring{B}\right),
\end{aligned}\right\}\end{align}
and second order Maxwell equations
\begin{align}\tag{2ndME}\label{2ndME}\left.\begin{aligned}
\partial_t^2 E-\Delta E &= -\partial_t j-\partial_x\rho,\\
E\left(0\right) &= \mathring{E},\\
\partial_t E\left(0\right) &= \left(\partial_{x_2}\mathring{B},-\partial_{x_1}\mathring{B}\right)-j\left(0\right),\\
\partial_t^2 B-\Delta B &= \partial_{x_1}j_2-\partial_{x_2}j_1,\\
B\left(0\right) &= \mathring{B},\\
\partial_t B\left(0\right) &= -\partial_{x_1}\mathring{E}_2+\partial_{x_2}\mathring{E}_1,
\end{aligned}\right\}\end{align}
are equivalent. Moreover, then also $\divg E=\rho$ globally in time.
\end{lemma}
We give a quite general condition that guarantees \eqref{LC}.
\begin{lemma}\label{LCco}
Let $g\in C$, and $f$, $d$, and $K$ of class $C^1$ with $\divg_pK=0$ and $f\left(t,x,\cdot\right)$ compactly supported for each $t\in\left[0,T\right]$ and $x\in\R^2$. Assume $\partial_tf+\widehat{p}\cdot\partial_xf+K\cdot\partial_pf=g$ and that $\int{g\,dp}=0$ holds. Then $\rho=\rho_f-\int_0^t{\divg_xd\,d\tau}$ and $j=j_f+d$ satisfy \eqref{LC}.
\end{lemma}
\begin{proof}
First, $\partial_t\left(-\int_0^t{\divg_xd\,d\tau}\right)+\divg_xd=0$ is obvious. Furthermore, integrating the Vlasov equation with respect to $p$ instantly yields $\partial_t\rho_f+\divg_xj_f=0$.
\end{proof}
Since \eqref{2ndME} consists of Cauchy problems for wave equations, we will need a solution formula for the 2D wave equation. In two dimensions, the (in $C^2$ unique) solution of the Cauchy problem
\begin{align*}
\partial_t^2 u-\Delta u&=f,\\
u\left(0\right)&=g,\\
\partial_tu\left(0\right)&=h,
\end{align*}
is given by the well known formula
\begin{align*}
u\left(t,x\right)&=\frac{1}{2\pi}\int_0^t{\int_{\left| x-y\right|<t-\tau}{\frac{f\left(\tau,y\right)}{\sqrt{\left(t-\tau\right)^2-\left| x-y\right|^2}}\,dyd\tau}}\\
&\phan+\frac{1}{2\pi}\int_{B_1}{\frac{g\left(x+ty\right)+t\nabla g\left(x+ty\right)\cdot y+th\left(x+ty\right)}{\sqrt{1-\left| y\right|^2}}\,dy}
\end{align*}
if the data are smooth.

\subsection{Control space for classical solutions}
In the following let $L>0$,
\begin{align*}
U\in V:=\left\{d\in W^{2,1}\left(0,T;C_b^4\left(\R^2;\R^2\right)\right)\mid d\left(t,x\right)=0\text{ for }\left| x\right|\geq L\right\},
\end{align*}
and let $V$ be equipped with the $W^{2,1}\left(0,T;C_b^4\left(\R^2;\R^2\right)\right)$-norm.
\section{Existence results}
\subsection{Estimates on the fields}\label{est}
\subsubsection{A generalized system}
The most important tool to get certain bounds is to have representations of the fields. One can use the solution formula for the wave equation and after some transformation of the integral expressions Gronwall-like estimates on the density and the fields can be derived. These bounds, for instance, will imply that the sequences constructed in Section \ref{loc} converge in a certain sense. Having that in mind it is useful not to work with the system \eqref{CVM} but with a somewhat generalized one with second order Maxwell equations:
\begin{align}\tag{GVM}\label{GVM}\left.\begin{aligned}
\partial_t f+\widehat{p}\cdot\partial_x f+\alpha\left(p\right)K\cdot\partial_p f &= g,\\
\partial_t^2 E-\Delta E &= -\partial_t j_f-\partial_t d-\partial_x\rho_f+\partial_x\int_0^t{\divg_xd\,d\tau},\\
\partial_t^2 B-\Delta B &= \partial_{x_1}j_{f,2}-\partial_{x_2}j_{f,1}+\partial_{x_1}d_2-\partial_{x_2}d_1,\\
\left(f,E,B\right)\left(0\right) &= \left(\mathring{f},\mathring{E},\mathring{B}\right),\\
\partial_t E\left(0\right) &= \left(\partial_{x_2}\mathring{B},-\partial_{x_1}\mathring{B}\right)-j_{\mathring{f}}-d\left(0\right),\\
\partial_t B\left(0\right) &= -\partial_{x_1}\mathring{E}_2+\partial_{x_2}\mathring{E}_1,
\end{aligned}\right\}\end{align}
with initial data $\mathring{f}$ of class $C^1_c$ and $\mathring{E}$, $\mathring{B}$ of class $C_b^2$. We assume that we already have functions $f$, $K$ of class $C^1$, $E$, $B$ of class $C^2$, $g$ of class $C_b$, $d$ of class $C^1\left(0,T;C_b^2\right)$ and $\alpha$ of class $C^1_b$ satisfying \eqref{GVM}. Furthermore we assume that $\divg_pK=0$ and that there is a $r>0$ such that $f\left(t,x,p\right)=g\left(t,x,p\right)=0$ if $\left| p\right|>r$.
\subsubsection{Estimates on the density}
\begin{lemma}\label{fest}
The density $f$ and its $\left(x,p\right)$-derivatives are estimated by
\begin{enumerate}[i)]
	\item \begin{align*}
	\left\|f\left(t\right)\right\|_\infty \leq \left\|\mathring{f}\right\|_\infty+\int_0^t{\left\|g\left(\tau\right)\right\|_\infty d\tau}
	\end{align*}
	if $g\in C$ and
	\item \begin{align*}
	\left\|\partial_{x,p}f\left(t\right)\right\|_\infty &\leq \left(\left\|\partial_{x,p}\mathring{f}\right\|_\infty+\int_0^t{\left\|\partial_{x,p}g\left(\tau\right)\right\|_\infty d\tau}\right)\\
&\phan\cdot\exp\left(\int_0^t{\left\|\partial_{x,p}\left(\alpha K\right)\left(\tau\right)\right\|_\infty d\tau}\right)
	\end{align*}
	if $g\in C^1$.
\end{enumerate}
\end{lemma}
\begin{proof}
This is easily proved by considering the characteristics of the Vlasov equation in \eqref{GVM}, which are defined via
\begin{align*}
\dot{X}=\widehat{P},\,\dot{P}=\alpha\left(P\right)K\left(s,X,P\right)
\end{align*}
with initial condition $\left(X,P\right)\left(t,t,x,p\right)=\left(x,p\right)$. Then 
\begin{align*}
f\left(t,x,p\right) &= \mathring{f}\left(\left(X,P\right)\left(0,t,x,p\right)\right)+\int_0^t{g\left(s,\left(X,P\right)\left(s,t,x,p\right)\right)ds}
\end{align*}
and, if $g\in C^1$,
\begin{align*}
\partial_{x,p}f\left(t,x,p\right) &= \left(\partial_{x,p}\mathring{f}\right)\left(\left(X,P\right)\left(0,t,x,p\right)\right)+\int_0^t{\left(\partial_{x,p}g\right)\left(s,\left(X,P\right)\left(s,t,z\right)\right)ds}\\
&\phan-\int_0^t{\left(\partial_{x,p}f\right)\left(s,\left(X,P\right)\left(s,t,z\right)\right)\left(\partial_{x,p}\left(\alpha K\right)\right)\left(s,\left(X,P\right)\left(s,t,z\right)\right)ds};
\end{align*}
see \cite[Sec. 5]{vkg}. The asserted estimates are hence straightforwardly derived.
\end{proof}
The $p$-support condition on $f$ is satisfied if $\supp\alpha\subset B_R$ for some $R>0$: Obviously for $\left| p\right|>\max\left\{R,r,r_0\right\}$ (where $\supp_p\mathring{f}\subset B_{r_0}$) we have $\dot{P}\left(s,t,x,p\right)=0$, hence $P\left(s,t,x,p\right)=p$ and therefore $\mathring{f}\left(\left(X,P\right)\left(0,t,x,p\right)\right)=g\left(s,\left(X,P\right)\left(s,t,x,p\right)\right)=0$.

In the following we denote by $C>0$ some generic constant that may change from line to line, but is only dependent on $T$, $r$, and $\alpha$ (i.e. its $C_b^1$-norm).
\subsubsection{Representation of the fields}
We can derive integral expressions for the fields $E$ and $B$ proceeding similarly to \cite{rvm1}.
\begin{lemma}\label{field}
We have $E=E^0+ES+ET+ED$ and $B=B^0+BS+BT+BD$ where $E^0$, $B^0$ are functionals of the initial data and $d\left(0\right)$, and where
\begin{align*}
ES_j &= \int_0^t{\int_{\left| x-y\right|<t-\tau}{\int{\frac{\left(\alpha\partial_p\left(es_j\right)+es_j\nabla\alpha\right)\cdot Kf+\left(es_j\right)g}{\sqrt{\left(t-\tau\right)^2-\left| x-y\right|^2}}\,dp}dy}d\tau},\\
BS &= \int_0^t{\int_{\left| x-y\right|<t-\tau}{\int{\frac{\left(\alpha\partial_p\left(bs\right)+bs\nabla\alpha\right)\cdot Kf+\left(bs\right)g}{\sqrt{\left(t-\tau\right)^2-\left| x-y\right|^2}}\,dp}dy}d\tau},\\
ET_j &= \int_0^t{\int_{\left| x-y\right|<t-\tau}{\int{\frac{et_j}{\left(t-\tau\right)\sqrt{\left(t-\tau\right)^2-\left| x-y\right|^2}}f\,dp}dy}d\tau},\\
BT &= \int_0^t{\int_{\left| x-y\right|<t-\tau}{\int{\frac{bt}{\left(t-\tau\right)\sqrt{\left(t-\tau\right)^2-\left| x-y\right|^2}}f\,dp}dy}d\tau},\\
ED_j &= -\frac{1}{2\pi}\int_0^t{\int_{\left| x-y\right|<t-\tau}{\frac{\partial_td_j-\int_0^\tau{\partial_{x_j}\divg_xd\,ds}}{\sqrt{\left(t-\tau\right)^2-\left| x-y\right|^2}}\,dy}d\tau},\\
BD &= \frac{1}{2\pi}\int_0^t{\int_{\left| x-y\right|<t-\tau}{\frac{\partial_{x_1}d_2-\partial_{x_2}d_1}{\sqrt{\left(t-\tau\right)^2-\left| x-y\right|^2}}\,dy}d\tau}.
\end{align*}
Furthermore the estimate
\begin{align*}
\left\|E\left(t\right)\right\|_\infty+\left\|B\left(t\right)\right\|_\infty&\leq C\left(\left\|\mathring{f}\right\|_\infty+\left\|\mathring{E}\right\|_{C_b^1}+\left\|\mathring{B}\right\|_{C_b^1}+\left\|d\right\|_{W^{1,1}\left(0,T;C_b^2\right)}\right)\\
&\phan+C\int_0^t{\left(\left(1+\left\|K\left(\tau\right)\right\|_\infty\right)\left\|f\left(\tau\right)\right\|_\infty+\left\|g\left(\tau\right)\right\|_\infty\right)d\tau}
\end{align*}
holds.\\
If additionally $\mathring{E}$, $\mathring{B}\in C_c$, and $d$ is compactly supported in $x$ uniformly in $t$, so are also the fields.
\end{lemma}
\begin{proof}
The representation formula are derived in much the same way as in \cite[Thm. 1]{rvm1}. The only difference is that here the source terms $g$ and $d$ appear. The support assertion is an immediate consequence of the representation formula. Physically, this is a result of the fact that electromagnetic fields can not propagate faster than the speed of light. Furthermore, the remaining estimate is a consequence of Remark \ref{fund}.
\end{proof}
\begin{remark}\label{remcomp}
If $f\left(t,x,\cdot\right)$ is compactly supported for every $t$, $x$, but not necessarily uniformly in $t$, $x$, nevertheless the fields are given by the formula above. For this, one does not need the uniformity. However, the estimate can not be obtained.
\end{remark}
\subsubsection{First derivatives of the fields}
The next step is to differentiate these representation formulas and deriving certain estimates. The method is similar to the previous one. The constant $C$ may now only depend on $T$, $r$, the initial data (i.e. their $C_b^2$-norms), and $\left\|\alpha\right\|_{C_b^1}$.
\begin{lemma}\label{fieldder}
If $g\in C^1$ and $d\in W^{2,1}\left(0,T;C_b^3\right)$, then the derivatives of the $S$-, $T$-, and $D$-terms are given by
\begin{align*}
\partial_{x_i}BS &= \int_0^t{\int_{\left| x-y\right|<t-\tau}{\int{\frac{\left(\alpha\partial_p\left(bs\right)+bs\nabla\alpha\right)\cdot\left(f\partial_{x_i}K+K\partial_{x_i}f\right)+bs\partial_{x_i}g}{\sqrt{\left(t-\tau\right)^2-\left| x-y\right|^2}}\,dp}dy}d\tau},\\
\partial_{x_i}BT &= \int_0^t{\int_{\left| x-y\right|<t-\tau}{\int{\frac{bt}{\left(t-\tau\right)\sqrt{\left(t-\tau\right)^2-\left| x-y\right|^2}}\partial_{x_i}f\,dp}dy}d\tau},\\
\partial_{x_i}BD &= \frac{1}{2\pi}\int_0^t{\int_{\left| x-y\right|<t-\tau}{\frac{\partial_{x_i}\partial_{x_1}d_2-\partial_{x_i}\partial_{x_2}d_1}{\sqrt{\left(t-\tau\right)^2-\left| x-y\right|^2}}\,dy}d\tau},\\
\partial_{x_i}ES &= \int_0^t{\int_{\left| x-y\right|<t-\tau}{\int{\frac{\left(\alpha\partial_p\left(es\right)+es\nabla\alpha\right)\cdot\left(f\partial_{x_i}K+K\partial_{x_i}f\right)+es\partial_{x_i}g}{\sqrt{\left(t-\tau\right)^2-\left| x-y\right|^2}}\,dp}dy}d\tau},\\
\partial_{x_i}ET &= \int_0^t{\int_{\left| x-y\right|<t-\tau}{\int{\frac{et}{\left(t-\tau\right)\sqrt{\left(t-\tau\right)^2-\left| x-y\right|^2}}\partial_{x_i}f\,dp}dy}d\tau},\\
\partial_{x_i}ED &= \frac{1}{2\pi}\int_0^t{\int_{\left| x-y\right|<t-\tau}{\frac{\partial_t\partial_{x_i}d-\int_0^\tau{\partial_{x_i}\partial_x\divg_xd\,ds}}{\sqrt{\left(t-\tau\right)^2-\left| x-y\right|^2}}\,dy}d\tau},\\
\partial_tBS &= \int_0^t{\int_{\left| x-y\right|<t-\tau}{\int{\frac{\left(\alpha\partial_p\left(bs\right)+bs\nabla\alpha\right)\cdot\left(f\partial_tK+K\partial_tf\right)+bs\partial_tg}{\sqrt{\left(t-\tau\right)^2-\left| x-y\right|^2}}\,dp}dy}d\tau}\\
&\phan+\int_{\left| x-y\right|<t}{\int{\frac{\left.\left(\alpha\partial_p\left(bs\right)+bs\nabla\alpha\right)\right|_{\tau=0}\cdot K\left(0\right)\mathring{f}+\left.bs\right|_{\tau=0}g\left(0\right)}{\sqrt{t^2-\left| x-y\right|^2}}\,dp}dy},\\
\partial_tBT &= \int_0^t{\int_{\left| x-y\right|<t-\tau}{\int{\frac{bt}{\left(t-\tau\right)\sqrt{\left(t-\tau\right)^2-\left| x-y\right|^2}}\partial_tf\,dp}dy}d\tau}\\
&\phan+\int_{\left| x-y\right|<t}{\int{\frac{\left.bt\right|_{\tau=0}}{t\sqrt{t^2-\left| x-y\right|^2}}\mathring{f}\,dp}dy},\\
\partial_tBD &= \frac{1}{2\pi}\int_0^t{\int_{\left| x-y\right|<t-\tau}{\frac{\partial_t\partial_{x_1}d_2-\partial_t\partial_{x_2}d_1}{\sqrt{\left(t-\tau\right)^2-\left| x-y\right|^2}}\,dy}d\tau}\\
&\phan+\frac{1}{2\pi}\int_{\left| x-y\right|<t}{\frac{\partial_{x_1}d_2\left(0\right)-\partial_{x_2}d_1\left(0\right)}{\sqrt{t^2-\left| x-y\right|^2}}\,dy},\\
\partial_tES &= \int_0^t{\int_{\left| x-y\right|<t-\tau}{\int{\frac{\left(\alpha\partial_p\left(es\right)+es\nabla\alpha\right)\cdot\left(f\partial_tK+K\partial_tf\right)+es\partial_tg}{\sqrt{\left(t-\tau\right)^2-\left| x-y\right|^2}}\,dp}dy}d\tau}\\
&\phan+\int_{\left| x-y\right|<t}{\int{\frac{\left.\left(\alpha\partial_p\left(es\right)+es\nabla\alpha\right)\right|_{\tau=0}\cdot K\left(0\right)\mathring{f}+\left.es\right|_{\tau=0}g\left(0\right)}{\sqrt{t^2-\left| x-y\right|^2}}\,dp}dy},\\
\partial_tET &= \int_0^t{\int_{\left| x-y\right|<t-\tau}{\int{\frac{et}{\left(t-\tau\right)\sqrt{\left(t-\tau\right)^2-\left| x-y\right|^2}}\partial_tf\,dp}dy}d\tau}\\
&\phan+\int_{\left| x-y\right|<t}{\int{\frac{\left.et\right|_{\tau=0}}{t\sqrt{t^2-\left| x-y\right|^2}}\mathring{f}\,dp}dy},\\
\partial_tED &= -\frac{1}{2\pi}\int_0^t{\int_{\left| x-y\right|<t-\tau}{\frac{\partial_t^2d-\partial_x\divg_xd}{\sqrt{\left(t-\tau\right)^2-\left| x-y\right|^2}}\,dy}d\tau}\\
&\phan-\frac{1}{2\pi}\int_{\left| x-y\right|<t}{\frac{\partial_td_j\left(0\right)}{\sqrt{t^2-\left| x-y\right|^2}}\,dy}.
\end{align*}
Furthermore the derivatives are estimated by
\begin{align*}
\left\|\partial_{t,x}E\left(t\right)\right\|_\infty+\left\|\partial_{t,x}B\left(t\right)\right\|_\infty&\leq C\left(1+\left\|K\right\|_\infty+\left\|f\right\|_\infty+\left\|g\right\|_\infty\right)\left(1+\left\|K\right\|_\infty\right)^2\\
&\phan \cdot\left(1+\ln_+\left(\left|\left\|\partial_{x,p}f\right\|\right|_{\left[0,t\right]}\right)+\int_0^t{\left\|\partial_{t,x,p}K\left(\tau\right)\right\|_\infty d\tau}\right)\\
&\phan +C\int_0^t{\left\|\partial_{t,x}g\left(\tau\right)\right\|_\infty d\tau}+C\left\|d\right\|_{W^{2,1}\left(0,T;C_b^3\right)}
\end{align*}
if $\left\|K\right\|_\infty<\infty$. Here $\left|\left\|a\right\|\right|_{\left[0,t\right]}:=\sup_{\substack{0\leq\tau\leq t}}{\left\|a\left(\tau\right)\right\|_\infty}$.
\end{lemma}
\begin{proof}
Similarly as before, this is proved by following \cite{rvm1}, now considering Theorem 3 there\-in.
\end{proof}

\subsection{A-priori bounds on the support with respect to $p$}
The most important property that is exploited later while showing global existence of a solution of \eqref{CVM}, is to have a-priori bounds on the $p$-support of $f$. This means: If we have a solution $\left(f,E,B\right)$ of \eqref{CVM} on $\left[0,T\right[$ with $f\in C^1$ and $E$, $B$ of class $C^2$, we have to show that
\begin{align*}
P\left(t\right):=\inf\left\{a>0\mid f\left(\tau,x,p\right)=0\text{ for all }\left| p\right|\geq a,\,0\leq\tau\leq t\right\}+3
\end{align*}
is controlled, i.e. $P\left(t\right)\leq Q$ for $0\leq t<T$ where $Q>0$ is some constant only dependent on $T$, the initial data (i.e. their $C_b^1$-norms and $P\left(0\right)$), $L$, and $\left\|U\right\|_V$. In the following the constants $C$ may also only depend on these numbers. Note that, per definition, $P$ is monotonically increasing and that $\left| f\right|\leq\left\|\mathring{f}\right\|_\infty$. Moreover, $P\left(t\right)<\infty$ for each $0\leq t<T$ because we have an a priori estimate on the $x$-support of $f$ via $\left|\dot{X}\right|\leq 1$, so that $\supp_xf\subset B_s$, and on the compact set $\left[0,t\right]\times\overline{B_s}$ the electromagnetic fields are bounded; hence the force field $E-\widehat{p}^\bot B$ is bounded there. Furthermore, \eqref{LC} holds by Lemma \ref{LCco}. Therefore and with Remark \ref{remcomp} we have the representations of the fields as given in Lemma \ref{field}. Moreover, we can also demand that $\left(f,E,B\right)$ solves
\begin{align}\tag{CVM2nd}\label{CVM2nd}\left.\begin{aligned}
\partial_t f+\widehat{p}\cdot\partial_x f+\left(E-\widehat{p}^\bot B\right)\cdot\partial_p f &= 0,\\
\partial_t^2 E-\Delta E &= -\partial_t j_f-\partial_t U-\partial_x\rho_f+\partial_x\int_0^t{\divg_xU\,d\tau},\\
\partial_t^2 B-\Delta B &= \partial_{x_1}j_{f,2}-\partial_{x_2}j_{f,1}+\partial_{x_1}U_2-\partial_{x_2}U_1,\\
\left(f,E,B\right)\left(0\right) &= \left(\mathring{f},\mathring{E},\mathring{B}\right),\\
\partial_t E\left(0\right) &= \left(\partial_{x_2}\mathring{B},-\partial_{x_1}\mathring{B}\right)-j_{\mathring{f}}-U\left(0\right),\\
\partial_t B\left(0\right) &= -\partial_{x_1}\mathring{E}_2+\partial_{x_2}\mathring{E}_1
\end{aligned}\right\}\end{align}
instead of \eqref{CVM} since both systems are equivalent by Lemma \ref{equi}.

We use the notation
\begin{align*}
\omega:=\frac{y-x}{\left| y-x\right|},\,a\wedge b:=a_1b_2-a_2b_1,\,K:=E-\widehat{p}^\bot B
\end{align*}
and follow \cite{rvm2}.
\subsubsection{Energy estimates}
The key in \cite{rvm2} is a sample of estimates that follow from the local energy conservation law
\begin{align*}
\partial_t\left(\frac{1}{2}\left| E\right|^2+\frac{1}{2}B^2+4\pi\int{f\sqrt{1+\left| p\right|^2}dp}\right)+\divg_x\left(-BE^\bot+4\pi\int{fp\,dp}\right)=0.
\end{align*}
However, this equation is false in our situation due to the external currents $U$. But still we are able to prove an analogue of \cite[Lem. 1]{rvm2}:
\begin{lemma}\label{rest}
Let $0\leq R\leq T$. The estimates
\begin{enumerate}[i)]
	\item \begin{align*}
	\sup_{\substack{x\in\R^2}}{\int_{\left| y-x\right|<R}{\left(\frac{1}{2}\left| E\right|^2+\frac{1}{2}B^2+4\pi\int{f\sqrt{1+\left| p\right|^2}dp}\right)dy}}\leq C,
	\end{align*}
	\item \begin{align*}
	\sup_{\substack{x\in\R^2}}\int_0^t\int_{\left| y-x\right|=t-\tau+R}&\left(\frac{1}{2}\left(E\cdot\omega\right)^2+\frac{1}{2}\left(B+\omega\wedge E\right)^2\right.\\
	& +\left.4\pi\int{f\sqrt{1+\left| p\right|^2}\left(1+\widehat{p}\cdot\omega\right)dp}\right)dS_yd\tau\leq C,
	\end{align*}
	\item \begin{align*}
	\sup_{\substack{x\in\R^2}}{\int_{\left| y-x\right|<R}{\rho_f^{\frac{3}{2}}\,dy}}\leq C,
	\end{align*}
	\item \begin{align*}
	\sup_{\substack{x\in\R^2}}{\int_{\left| y-x\right|<R}{\left(\int{\frac{f}{\sqrt{1+\left| p\right|^2}}\,dp}\right)^3 dy}}\leq C
	\end{align*}
\end{enumerate}
hold for all $t\in\left[0,T\right[$.
\end{lemma}
\begin{proof}
We split the electro-magnetic fields into internal and external fields; precisely, they are defined by
\begin{align*}
\partial_t E_{\inte,1}-\partial_{x_2}B_{\inte} &= -j_{f,1},\\
\partial_t E_{\inte,2}+\partial_{x_1}B_{\inte} &= -j_{f,2},\\
\partial_t B_{\inte}+\partial_{x_1}E_{\inte,2}-\partial_{x_2}E_{\inte,1} &= 0,\\
\left(E_{\inte},B_{\inte}\right)\left(0\right) &= \left(\mathring{E},\mathring{B}\right)
\end{align*}
and
\begin{align*}
\partial_t E_{\ext,1}-\partial_{x_2}B_{\ext} &= -U_1,\\
\partial_t E_{\ext,2}+\partial_{x_1}B_{\ext} &= -U_2,\\
\partial_t B_{\ext}+\partial_{x_1}E_{\ext,2}-\partial_{x_2}E_{\ext,1} &= 0,\\
\left(E_{\ext},B_{\ext}\right)\left(0\right) &= 0.
\end{align*}
Indeed, the existence of $\left(E_{\ext},B_{\ext}\right)$ is guaranteed since the (time evolutionary) Max\-well equations form a linear, symmetric, hyperbolic system, see \cite[Thm. I]{kato}. Because of $U\in V$ we have $E_{\ext}$, $B_{\ext}\in C\left(0,T;H^3\right)\cap C^1\left(0,T;H^2\right)\subset C^1$; furthermore
\begin{align*}
\left\|\left(E_{\ext},B_{\ext}\right)\left(t\right)\right\|_{\infty}\leq C\left\|\left(E_{\ext},B_{\ext}\right)\left(t\right)\right\|_{H^2}\leq C\int_0^T{\left\|U\left(\tau\right)\right\|_{H^2}d\tau}\leq C\left\|U\right\|_V=C
\end{align*}
by Sobolev's embedding theorem and the support condition on $U$. Because of the linearity of the Maxwell equations it holds that $E_{\inte}:=E-E_{\ext}$ and $B_{\inte}:=B-B_{\ext}$ solve their equations mentioned earlier and are of class $C^1$. Now let
\begin{align*}
e_{\inte}:=\frac{1}{2}\left| E_{\inte}\right|^2+\frac{1}{2}B_{\inte}^2+4\pi\int{f\sqrt{1+\left| p\right|^2}dp}
\end{align*}
which is physically the energy density of the internal system and
\begin{align*}
e:=\frac{1}{2}\left| E\right|^2+\frac{1}{2}B^2+4\pi\int{f\sqrt{1+\left| p\right|^2}dp}.
\end{align*}
We have
\begin{align*}
& \partial_te_{\inte}+\divg_x\left(-B_{\inte}E_{\inte}^\bot+4\pi\int{fp\,dp}\right)\\
&= E_{\inte}\cdot\partial_tE_{\inte}+B_{\inte}\partial_tB_{\inte}+4\pi\int{\partial_tf\sqrt{1+\left| p\right|^2}dp}+E_{\inte,2}\partial_{x_1}B_{\inte}+B_{\inte}\partial_{x_1}E_{\inte,2}\\
&\phan -E_{\inte,1}\partial_{x_2}B_{\inte}-B_{\inte}\partial_{x_2}E_{\inte,1}+4\pi\int{\partial_xf\cdot p\,dp}\\
&= -E_{\inte}\cdot j_f-4\pi\int{K\cdot\partial_pf\sqrt{1+\left| p\right|^2}dp}\\
&= -E_{\inte}\cdot j_f+4\pi E\cdot\int{f\partial_p\sqrt{1+\left| p\right|^2}dp}+4\pi B\int{f\divg_pp^\bot dp}\\
&= E_{\ext}\cdot j_f
\end{align*}
where we made use of the respective Vlasov-Maxwell equations, $\partial_p\sqrt{1+\left| p\right|^2}=\widehat{p}$, and $\divg_pp^\bot=0$. We integrate this identity over a suitable set and arrive at
\begin{align*}
& \int_0^t{\int_{\left| y-x\right|<t-\tau+R}{E_{\ext}\cdot j_f\,dy}d\tau}\\
&= \int_0^t{\int_{\left| y-x\right|<t-\tau+R}{\left(\partial_\tau e_{\inte}+\divg_y\left(-B_{\inte}E_{\inte}^\bot+4\pi\int{fpdp}\right)\right)dy}d\tau}\\
&= -\int_{\left| y-x\right|<t+R}{e_{\inte}\left(0,y\right)dy}+\int_{\left| y-x\right|<R}{e_{\inte}\left(t,y\right)dy}\\
&\phan +\frac{1}{\sqrt{2}}\int_0^t{\int_{\left| y-x\right|=t-\tau+R}{\left(e_{\inte}+\omega\cdot\left(-B_{\inte}E_{\inte}^\bot+4\pi\int{fpdp}\right)\right)dS_y}d\tau}\numb\label{e0}
\end{align*}
after an integration by parts in $\left(\tau,y\right)$. The integrand of the last integral is non-negative because of
\begin{align*}
0&\leq d_{\inte}:=\frac{1}{2}\left(E_{\inte}\cdot\omega\right)^2+\frac{1}{2}\left(B_{\inte}+\omega\wedge E_{\inte}\right)^2+4\pi\int{f\sqrt{1+\left| p\right|^2}\left(1+\widehat{p}\cdot\omega\right)dp}\\
&= \frac{1}{2}E_{\inte,1}^2\omega_1^2+\frac{1}{2}E_{\inte,2}^2\omega_2^2+\frac{1}{2}B_{\inte}^2+B_{\inte}\omega_1E_{\inte,2}-B_{\inte}\omega_2E_{\inte,1}+\frac{1}{2}E_{\inte,2}^2\omega_1^2\\
&\phan +\frac{1}{2}E_{\inte,1}^2\omega_2^2+4\pi\int{f\sqrt{1+\left| p\right|^2}dp}+\omega\cdot 4\pi\int{fp\,dp}\\
&= \frac{1}{2}E_{\inte,1}^2+\frac{1}{2}E_{\inte,2}^2+\frac{1}{2}B_{\inte}^2+4\pi\int{f\sqrt{1+\left| p\right|^2}dp}+\omega_1B_{\inte}E_{\inte,2}-\omega_2B_{\inte}E_{\inte,1}\\
& \phan+\omega\cdot 4\pi\int{fp\,dp}\\
&= e_{\inte}+\omega\cdot\left(-B_{\inte}E_{\inte}^\bot+4\pi\int{fp\,dp}\right)\numb\label{int0};
\end{align*}
note that $1+\widehat{p}\cdot\omega\geq 1-1\cdot 1=0$ and $\left|\omega\right|=1$. The left hand side of \eqref{e0} has to be investigated. The external fields are bounded by $C$, hence
\begin{align}\label{source}
\left|\int_0^t{\int_{\left| y-x\right|<t-\tau+R}{E_{\ext}\cdot j_f\,dy}d\tau}\right|\leq C\int_0^t{\left\|j_f\left(\tau\right)\right\|_{L^1}d\tau}\leq C\int_0^t{\left\|\rho_f\left(\tau\right)\right\|_{L^1}d\tau}\leq C
\end{align}
since the $L^1$-norm of $\rho_f$ is constant in time.

Now we can prove the assertions using \eqref{e0}, \eqref{int0}, and \eqref{source}:
\begin{enumerate}[i)]
	\item We have
	\begin{align*}
	\int_{\left| y-x\right|<R}{e_{\inte}\,dy}\leq\int_{\left| y-x\right|<t+R}{e_{\inte}\left(0,y\right)dy}+C\leq C\left(R+t\right)^2+C\leq C
	\end{align*}
	since $t,R\leq T$. Together with
	\begin{align*}
	e\leq 2e_{\inte}+\left| E_{\ext}\right|^2+\left| B_{\ext}\right|^2\leq 2e_{\inte}+C
	\end{align*}
	we conclude
	\begin{align*}
	\int_{\left| y-x\right|<R}{e\,dy}\leq C+CR^2\leq C.
	\end{align*}
	\item Similarly,
	\begin{align*}
	\int_0^t{\int_{\left| y-x\right|=t-\tau+R}{d_{\inte}\,dS_y}d\tau}\leq\sqrt{2}\int_{\left| y-x\right|<t+R}{e_{\inte}\left(0,y\right)dy}+C\leq C
	\end{align*}
	and
	\begin{align*}
	d&:=\frac{1}{2}\left(E\cdot\omega\right)^2+\frac{1}{2}\left(B+\omega\wedge E\right)^2+4\pi\int{f\sqrt{1+\left| p\right|^2}\left(1+\widehat{p}\cdot\omega\right)dp}\\
	&\leq 2d_{\inte}+2\left| E_{\ext}\right|^2+\left| B_{\ext}\right|^2\leq 2d_{\inte}+C
	\end{align*}
	yield
	\begin{align*}
	\int_0^t{\int_{\left| y-x\right|=t-\tau+R}{d\,dS_y}d\tau}\leq C+Ct\left(t+R\right)^2\leq C.
	\end{align*}
	\item For $r>0$ it holds that
	\begin{align*}
	\rho_f&= 4\pi\int{f\,dp}=4\pi\int_{\left| p\right|<r}{f\,dp}+4\pi\int_{\left| p\right|\geq r}{f\,dp}\\
	&\leq Cr^2+4\pi r^{-1}\int_{\left| p\right|\geq r}{f\sqrt{1+\left| p\right|^2}dp}\leq C\left(r^2+r^{-1}e\right).
	\end{align*}
	Now choose $r:=e^{\frac{1}{3}}>0$ to derive $\rho_f\leq Ce^{\frac{2}{3}}$ (if $e=0$ then also $\rho_f=0$) and hence
	\begin{align*}
	\int_{\left| y-x\right|<R}{\rho_f^{\frac{3}{2}}\,dy}\leq C\int_{\left| y-x\right|<R}{e\,dy}\leq C.
	\end{align*}
	\item Similarly,
	\begin{align*}
	\int{\frac{f}{\sqrt{1+\left| p\right|^2}}\,dp}&\leq C\int_{\left| p\right|<r}{\frac{1}{\sqrt{1+\left| p\right|^2}}\,dp}+\frac{1}{1+r^2}\int_{\left| p\right|\geq r}{f\sqrt{1+\left| p\right|^2}dp}\\
	&\leq C\int_0^r{\frac{s}{\sqrt{1+s^2}}\,ds}+\frac{1}{r^2}e\leq C\left(r+r^{-2}e\right)\leq Ce^{\frac{1}{3}}
	\end{align*}
	for again $r:=e^{\frac{1}{3}}$ which yields
	\begin{align*}
	\int_{\left| y-x\right|<R}{\left(\int{\frac{f}{\sqrt{1+\left| p\right|^2}}\,dp}\right)^3 dy}\leq C\int_{\left| y-x\right|<R}{e\,dy}\leq C.
	\end{align*}
\end{enumerate}
\end{proof}

\subsubsection{Estimates on the fields}
The crucial problem is to estimate the fields in a proper way. To this end, we use the representation formula stated in Lemma \ref{field}. Unfortunately, the estimates there can not be applied because, of course, we can not assume that $P\left(t\right)$ is controlled. 
\begin{lemma}
We have
\begin{align*}
\left| ES_1\right|+\left| ES_2\right|+\left| BS\right|&\leq CP\left(t\right)\ln P\left(t\right)+C\int_0^t{\left(\left\|E\left(\tau\right)\right\|_\infty+\left\|B\left(\tau\right)\right\|_\infty\right)d\tau},\\
\left| ET_1\right|+\left| ET_2\right|+\left| BT\right|&\leq CP\left(t\right)\ln P\left(t\right),\\
\left| ED\right|,\left| BD\right|&\leq C\left\|U\right\|_{W^{1,1}\left(0,T;C_b^2\right)}\leq C,\\
\left| E^0\right|,\left| B^0\right|&\leq C.
\end{align*}
\end{lemma}
\begin{proof}
The estimates on the $S$-and $T$-terms are derived in much the same way as in \cite[Sec. 2]{rvm2}. Note that the energy estimates of Lemma \ref{rest}, that had to be modified in our situation, are enough to carry out the the proofs therein.

The estimate on the $D$-terms is derived straightforwardly, as well as the estimate on $E^0$, $B^0$, the latter parts only containing terms of the initial data and $U\left(0\right)$.
\end{proof}
Now we can finally prove:
\begin{lemma}\label{Q}
The a-priori bound $P\left(t\right)\leq Q$ holds, where $Q$ only depends on $T$, the $C_b^1$-norms of the initial data, $\supp_p\mathring{f}$ (which basically coincides with $P\left(0\right)$), L, and $\left\|U\right\|_V$.
\end{lemma}
\begin{proof}
Collecting all bounds on the fields we arrive at
\begin{align*}
\left\|E\left(t\right)\right\|_\infty+\left\|B\left(t\right)\right\|_\infty\leq C+CP\left(t\right)\ln P\left(t\right)+C\int_0^t{\left(\left\|E\left(\tau\right)\right\|_\infty+\left\|B\left(\tau\right)\right\|_\infty\right)d\tau}.
\end{align*}
As in \cite{rvm2}, this is enough to show that
\begin{align*}
P\left(t\right)\leq C+C\int_0^t{P\left(s\right)\ln P\left(s\right)ds},
\end{align*}
from which the assertion follows immediately.
\end{proof}

\subsection{Existence of classical solutions}\label{loc}
\subsubsection{The iteration scheme}
In the following we want to construct a solution of \eqref{CVM}. We will only sketch the main ideas, since similar procedures have already been carried out in the literature, see for example \cite[Sec. V]{sf}.

We work with initial data $\mathring{f}\geq 0$ of class $C^2_c$, $\mathring{E}$, $\mathring{B}$ of class $C_b^3$, and control $U\in V$ that satisfy \eqref{CC}, i.e. $\divg\mathring{E}=\rho_{\mathring{f}}$. We have to approximate these functions, so let $\mathring{f}_k\rightarrow\mathring{f}$ in $C_b^2$, $\mathring{E}_k\rightarrow\mathring{E}$ and $\mathring{B}_k\rightarrow\mathring{B}$ in $C_b^3$ with $\mathring{f}_k\in C^\infty_c$, $\mathring{E}_k$, $\mathring{B}_k\in C^\infty$, and furthermore $U_k\rightarrow U$ in $V$ with $U_k\in C^\infty$ (note that $C^\infty$ is dense in $V$).

The strategy to obtain a solution of \eqref{CVM} is the following: By iteration we construct densities $f_k$ and fields $E_k$, $B_k$ in such a way that these functions will converge in a proper sense and that we may pass to the limit in \eqref{CVM}. However, it is more convenient to work with a modified system. As the previous section suggests, it is crucial to control the $p$-support of $f$. For this reason we first consider a cut-off system on $\left[0,T\right]$ where we modify the original Vlasov equation and use the second order Maxwell equations (\eqref{CC} and \eqref{LC} need not hold for the iterates):
\begin{align}\tag{$\alpha$VM}\label{aVM}\left.\begin{aligned}
\partial_t f+\widehat{p}\cdot\partial_x f+\alpha\left(p\right)\left(E-\widehat{p}^\bot B\right)\cdot\partial_p f &= 0,\\
\partial_t^2 E-\Delta E &= -\partial_t j_f-\partial_t U-\partial_x\rho_f+\partial_x\int_0^t{\divg_xU\,d\tau},\\
\partial_t^2 B-\Delta B &= \partial_{x_1}j_{f,2}-\partial_{x_2}j_{f,1}+\partial_{x_1}U_2-\partial_{x_2}U_1,\\
\left(f,E,B\right)\left(0\right) &= \left(\mathring{f},\mathring{E},\mathring{B}\right),\\
\partial_t E\left(0\right) &= \left(\partial_{x_2}\mathring{B},-\partial_{x_1}\mathring{B}\right)-j_{\mathring{f}}-U\left(0\right),\\
\partial_t B\left(0\right) &= -\partial_{x_1}\mathring{E}_2+\partial_{x_2}\mathring{E}_1.
\end{aligned}\right\}\end{align}
Here, let the cut-off function $\alpha$ be of class $C^\infty_c\left(\R^2\right)$ with $\alpha\left(p\right)=1$ for $\left| p\right|\leq 2Q$. The property of the constant $Q$ will imply that a solution of \eqref{aVM} is also a solution of \eqref{CVM}. 

We start the iteration with $f_0\left(t,x,p\right):=\mathring{f}_0\left(x,p\right)$, $E_0\left(t,x\right):=\mathring{E}_0\left(x\right)$, $B_0\left(t,x,p\right):=\mathring{B}_0\left(x\right)$. The induction hypothesis is that $f_k$, $E_k$, and $B_k$ are of class $C^\infty$ and that the fields are bounded. Given $f_{k-1}$, $E_{k-1}$, and $B_{k-1}$, we firstly define $f_k$ as the solution of
\begin{align*}
\partial_t f_k+\widehat{p}\cdot\partial_x f_k+\alpha\left(p\right)\left(E_{k-1}-\widehat{p}^\bot B_{k-1}\right)\cdot\partial_p f_k &= 0,\\
f_k\left(0\right) &= \mathring{f}_k,
\end{align*}
namely
\begin{align*}
f_k\left(t,x,p\right)=\mathring{f}_k\left(X_k\left(0,t,x,p\right),P_k\left(0,t,x,p\right)\right)
\end{align*}
with the characteristics defined by
\begin{align*}
\dot{X}_k&=\widehat{P}_k,&X_k\left(t,t,x,p\right)&=x,\\
\dot{P}_k&=\alpha\left(P_k\right)\left(E_{k-1}-\widehat{P}_k^\bot B_{k-1}\right)\left(s,X_k\right),&P_k\left(t,t,x,p\right)&=p.
\end{align*}
We conclude that $X_k$ and $P_k$ are of class $C^\infty$ in all four variables by the induction hypothesis. This yields that even $f_k\in C^\infty$. Since $\alpha$ is compactly supported the $p$-support of $f_k$ is controlled by a constant $C$. Hence, $\rho_{f_k}$ and $j_{f_k}$ are well defined as $C^\infty\cap C_b^1$-functions.
 
Secondly, we define $E_k$ and $B_k$ as the solution of
\begin{align*}
\partial_t^2 E_k-\Delta E_k &= -\partial_t j_{f_k}-\partial_t U_k-\partial_x\rho_{f_k}+\partial_x\int_0^t{\divg_xU_k\,d\tau},\\
\partial_t^2 B_k-\Delta B_k &= \partial_{x_1}j_{f_k,2}-\partial_{x_2}j_{f_k,1}+\partial_{x_1}U_{k,2}-\partial_{x_2}U_{k,1},\\
\left(E_k,B_k\right)\left(0\right) &= \left(\mathring{E}_k,\mathring{B}_k\right),\\
\partial_t E_k\left(0\right) &= \left(\partial_{x_2}\mathring{B}_k,-\partial_{x_1}\mathring{B}_k\right)-j_{\mathring{f}_k}-U_k\left(0\right),\\
\partial_t B_k\left(0\right) &= -\partial_{x_1}\mathring{E}_{k,2}+\partial_{x_2}\mathring{E}_{k,1}.
\end{align*}
Indeed, we can solve these wave equations by applying the solution formula for the wave equation. Since the right hand sides of the above equations are of class $C^\infty$ and bounded, so are also $E_k$ and $B_k$. Applying Lemmas \ref{fest}, \ref{field}, and \ref{fieldder} then shows that the iterates are bounded in $C_b^1$.

As for the second derivatives, we differentiate $\eqref{aVM}$ and have, for example,
\begin{align}\label{SD}\left.\begin{aligned}
\partial_t\partial_{x_i}f_k+\widehat{p}\cdot\partial_x\partial_{x_i}f_k+&\\
\alpha K_{k-1}\cdot\partial_p\partial_{x_i}f_k &= -\alpha\partial_{x_i}K_{k-1}\cdot\partial_pf_k,\\
\partial_t^2\partial_{x_i} E_k-\Delta\partial_{x_i} E_k &= -\partial_t j_{\partial_{x_i}f_k}-\partial_t\partial_{x_i} U_k-\partial_x\rho_{\partial_{x_i}f_k}+\partial_x\int_0^t{\divg_x\partial_{x_i}U_k\,d\tau},\\
\partial_t^2\partial_{x_i} B_k-\Delta\partial_{x_i} B_k &= \partial_{x_1}j_{\partial_{x_i}f_k,2}-\partial_{x_2}j_{\partial_{x_i}f_k,1}+\partial_{x_1}\partial_{x_i}U_{k,2}-\partial_{x_2}\partial_{x_i}U_{k,1},\\
\left(\partial_{x_i}f_k,\partial_{x_i}E_k,\partial_{x_i}B_k\right)\left(0\right) &= \left(\partial_{x_i}\mathring{f}_k,\partial_{x_i}\mathring{E}_k,\partial_{x_i}\mathring{B}_k\right),\\
\partial_t\partial_{x_i} E_k\left(0\right) &= \left(\partial_{x_2}\partial_{x_i}\mathring{B}_k,-\partial_{x_1}\partial_{x_i}\mathring{B}_k\right)-j_{\partial_{x_i}\mathring{f}_k}-\partial_{x_i}U_k\left(0\right),\\
\partial_t\partial_{x_i} B_k\left(0\right) &= -\partial_{x_1}\partial_{x_i}\mathring{E}_{k,2}+\partial_{x_2}\partial_{x_i}\mathring{E}_{k,1}
\end{aligned}\right\}\end{align}
and then apply the estimates of Lemmas \ref{field} and \ref{fieldder}. Note that for this we need four space derivatives in the definition of $V$ so that $\left\|\partial_xU_k\right\|_{W^{2,1}\left(0,T;C_b^3\right)}$ is bounded. Likewise, one proceeds with the other second order derivatives. Altogether, the iterates are bounded in $C_b^2$.

After that, considering the difference of the iterates of the $k$-th step and the $l$-th step, Lemmas \ref{fest}, \ref{field}, and \ref{fieldder} yield that the iteration sequences are even Cauchy sequences in $C_b^1$, so that they converge to some $\left(f,E,B\right)$ in the $C_b^1$-norm.

For later considerations it will be convenient that the density and the fields are even $C_b^2$. Since all second derivatives are bounded in $L^\infty\left(\left[0,T\right]\times\R^j\right)$ ($j=4$ or $2$ respectively) they converge, after extracting a suitable subsequence, in the weak-*-sense. Of course, these limits have to be the respective weak derivatives of $f$, $E$, and $B$. The remaining part is to show that the weak derivatives just obtained are in fact classical ones. For this sake, have a look at the representation formula for $\partial_{x_i}\partial_{x_j}B_k$; use system \eqref{SD} and Lemma \ref{fieldder}:
\begin{align*}
& \partial_{x_i}\partial_{x_j}B_k-\partial_{x_i}\overline{B}^0_k\\
&= \int_0^t{\int_{\left| x-y\right|<t-\tau}{\int{\frac{bt}{\left(t-\tau\right)\sqrt{\left(t-\tau\right)^2-\left| x-y\right|^2}}\partial_{x_i}\partial_{x_j}f_k\,dp}dy}d\tau}\\
&\phan+ \int_0^t{\int_{\left| x-y\right|<t-\tau}{\int{\frac{\left(\alpha\partial_p\left(bs\right)+bs\nabla\alpha\right)\cdot \partial_{x_j}f_k\partial_{x_i}K_{k-1}}{\sqrt{\left(t-\tau\right)^2-\left| x-y\right|^2}}\,dp}dy}d\tau}\\
&\phan+ \int_0^t{\int_{\left| x-y\right|<t-\tau}{\int{\frac{\left(\alpha\partial_p\left(bs\right)+bs\nabla\alpha\right)\cdot K_{k-1}\partial_{x_i}\partial_{x_j}f_k}{\sqrt{\left(t-\tau\right)^2-\left| x-y\right|^2}}\,dp}dy}d\tau}\\
&\phan- \int_0^t{\int_{\left| x-y\right|<t-\tau}{\int{\frac{\left(bs\right)\alpha\partial_{x_i}\partial_{x_j}K_{k-1}\cdot\partial_pf_k}{\sqrt{\left(t-\tau\right)^2-\left| x-y\right|^2}}\,dp}dy}d\tau}\\
&\phan- \int_0^t{\int_{\left| x-y\right|<t-\tau}{\int{\frac{\left(bs\right)\alpha\partial_{x_j}K_{k-1}\cdot\partial_{x_i}\partial_pf_k}{\sqrt{\left(t-\tau\right)^2-\left| x-y\right|^2}}\,dp}dy}d\tau}\\
&\phan+ \frac{1}{2\pi}\int_0^t{\int_{\left| x-y\right|<t-\tau}{\frac{\partial_{x_1}\partial_{x_i}U_{k,2}-\partial_{x_2}\partial_{x_i}U_{k,2}}{\sqrt{\left(t-\tau\right)^2-\left| x-y\right|^2}}\,dy}d\tau}.
\end{align*}
Here, $\overline{B}^0_k$ is the '$B^0$' of system \eqref{SD} and converges to the respective expression without indices.

We are allowed to pass to the limit in the integral expressions because all kernels are integrable, $\left(f_k,E_k,B_k\right)$ converge in $C_b^1$, the second derivatives weak-* in $L^\infty$, and $U_k$ in $V$. Hence we can omit the indices in the equation above or equivalently
\begin{align*}
& \partial_{x_i}\partial_{x_j}B-\partial_{x_i}\overline{B}^0\\
&= \int_0^t{\int_{\left| x-y\right|<t-\tau}{\int{\frac{bt}{\left(t-\tau\right)\sqrt{\left(t-\tau\right)^2-\left| x-y\right|^2}}\partial_{x_i}\partial_{x_j}f\,dp}dy}d\tau}\\
&\phan+ \int_0^t{\int_{\left| x-y\right|<t-\tau}{\int{\frac{\left(\alpha\partial_p\left(bs\right)+bs\nabla\alpha\right)\cdot \partial_{x_i}\left(K\partial_{x_j}f\right)}{\sqrt{\left(t-\tau\right)^2-\left| x-y\right|^2}}\,dp}dy}d\tau}\\
&\phan- \int_0^t{\int_{\left| x-y\right|<t-\tau}{\int{\frac{\left(bs\right)\alpha\partial_{x_i}\left(\partial_{x_j}K\cdot\partial_pf\right)}{\sqrt{\left(t-\tau\right)^2-\left| x-y\right|^2}}\,dp}dy}d\tau}\\
&\phan+ \frac{1}{2\pi}\int_0^t{\int_{\left| x-y\right|<t-\tau}{\frac{\partial_{x_1}U_2-\partial_{x_2}U_1}{\sqrt{\left(t-\tau\right)^2-\left| x-y\right|^2}}\,dy}d\tau}
\end{align*}
and conclude that $\partial_{x_i}\partial_{x_j}B$ is continuous which is an immediate consequence of $U\in V$ and the following lemma:
\begin{lemma}\label{intcont}
Denote $M:=\left\{\left(s,z\right)\in\left[0,T\right]\times\R^n\mid 0\leq s\leq T,\,\left| z\right|<s\right\}$ and let\linebreak $h\in C\left(\left[0,T\right]\times\R^{n+m}\right)$ with uniform support in $p\in\R^m$, i.e. $\supp_ph\subset B_r$ for some $r>0$, and let $w\in C^1\left(M\times B_r\right)$ and $\gamma\in\left\{t,x_1,\dots x_n\right\}$. Furthermore let one of the following options hold:
\begin{enumerate}[i)]
\item $h\in W^{1,\infty}\left(\left[0,T\right]\times\R^{n+m}\right)$ and $w\in L^1\left(M\times B_r\right)$,
\item $h\in W^{1,1}\left(0,T;L^\infty\left(\R^{n+m}\right)\right)$ if $\gamma=t$ or $h\in L^\infty\left(0,T;W^{1,\infty}\left(\R^{n+m}\right)\right)$ if $\gamma=x_i$ respectively, and
\begin{align*}
\int_{s-d<\left| z\right|<s}{\int_{B_r}{\left| w\left(s,z,p\right)\right| dp}dz}\rightarrow 0
\end{align*}
for $d\rightarrow 0$ uniformly in $s\in\left[0,T\right]$.
\end{enumerate}
Then
\begin{align*}
H\left(t,x\right) &:= \int_0^t{\int_{\left| x-y\right|<t-\tau}{\int{\left(\partial_\gamma h\right)\left(\tau,y,p\right)w\left(t-\tau,y-x,p\right)dp}dy}d\tau}\\
&= \int_0^t{\int_{\left| z\right|<s}{\int{\left(\partial_\gamma h\right)\left(t-s,x+z,p\right)w\left(s,z,p\right)dp}dz}ds}
\end{align*}
is continuous in $\left(t,x\right)\in\left[0,T\right]\times\R^n$.
\end{lemma}
\begin{proof}
Let $\gamma=x_i$ and $\epsilon>0$ be given. For $\left(t,x\right)\in\left[0,T\right]\times\R^n$ and $d>0$ define
\begin{align*}
I_d\left(t,x\right):=\int_0^t{\int_{s-d<\left| z\right|<s}{\int{\left(\partial_{x_i}h\right)\left(t-s,x+z,p\right)w\left(s,z,p\right)dp}dz}ds}
\end{align*}
and estimate in case i)
\begin{align*}
\left| I_d\left(t,x\right)\right|\leq\left\|\partial_{x_i}h\right\|_\infty\int_0^T{\int_{s-d<\left| z\right|<s}{\int_{B_r}{\left| w\left(s,z,p\right)\right| dp}dz}ds}\rightarrow0
\end{align*}
and in case ii)
\begin{align*}
\left| I_d\left(t,x\right)\right|\leq\int_0^T{\left\|\partial_{x_i}h\left(s\right)\right\|_\infty ds}\left\|s\mapsto\int_{s-d<\left| z\right|<s}{\int_{B_r}{\left| w\left(s,z,p\right)\right| dp}dz}\right\|_\infty\rightarrow0
\end{align*}
for $d\rightarrow0$ uniformly in $\left(t,x\right)$. Thus we can choose $d$ so that $\left| I_d\left(t,x\right)\right|<\frac{\epsilon}{4}$ for all $\left(t,x\right)$. For now fixed $d$ consider the remaining integral and integrate by parts
\begin{align*}
J_d\left(t,x\right) &:= \int_0^t{\int_{\left| z\right|<s-d}{\int{\left(\partial_{x_i}h\right)\left(t-s,x+z,p\right)w\left(s,z,p\right)dp}dz}ds}\\
&= \int_0^t{\int_{\left| z\right|<s-d}{\int{\left(\partial_{z_i}h\right)\left(t-s,x+z,p\right)w\left(s,z,p\right)dp}dz}ds}\\
&= -\int_0^t{\int_{\left| z\right|<s-d}{\int{h\left(t-s,x+z,p\right)\partial_{z_i}w\left(s,z,p\right)dp}dz}ds}\\
&\phan+ \int_0^t{\int_{\left| z\right|=s-d}{\int{h\left(t-s,x+z,p\right)w\left(s,z,p\right)\frac{1}{\sqrt{2}}\,dp}dS_z}ds}\\
&\phan+ \int_{\left| z\right|<t-d}{\int{h\left(0,x+z,p\right)w\left(t,z,p\right)dp}dz}.
\end{align*}
This is allowed because the integration domain is away from the possibly singular set $\left| z\right|=s$. For that very reason $J_d$ is obviously continuous by the standard theorem for parameter integrals, so if $\left(\delta t,\delta x\right)$ is small enough (with $t+\delta t\in\left[0,T\right]$) we have
\begin{align*}
\left| J_d\left(t+\delta t,x+\delta x\right)-J_d\left(t,x\right)\right|<\frac{\epsilon}{2}.
\end{align*}
Finally with $H=I_d+J_d$ we conclude
\begin{align*}
& \left| H\left(t+\delta t,x+\delta x\right)-H\left(t,x\right)\right|\\
&\leq\left| I_d\left(t+\delta t,x+\delta x\right)\right|+\left| I_d\left(t,x\right)\right|+\left| J_d\left(t+\delta t,x+\delta x\right)-J_d\left(t,x\right)\right|<\epsilon.
\end{align*}
Analogously, one proves the assertion for $\gamma=t$.
\end{proof}
This lemma is applicable since $f$ has uniform support in $p$, $\partial_xf$, $\partial_pf$, and $\partial_xK$ are of class $W^{1,\infty}$, $\left| bs\right|$, $\left| bt\right|\leq C\left(r\right)$, and by Remark \ref{fund}. Next, we have a representation formula for $\partial_t\partial_{x_j}B_k$ according to Lemma \ref{fieldder}. Analogously we conclude that $\partial_t\partial_{x_j}B$ is continuous. For this, note that the terms without an $\int_0^t$-integral are easy to handle since there only initial values appear.

The procedure for $E$ is nearly the same. The only critical point is to ensure that
\begin{align*}
\int_0^t{\int_{\left| x-y\right|<t-\tau}{\frac{\partial_t^2\partial_{x_j}U}{\sqrt{\left(t-\tau\right)^2-\left| x-y\right|^2}}\,dy}d\tau}
\end{align*}
is continuous for $U\in V$. To this end, we can apply Lemma \ref{intcont} with $h=\partial_t\partial_{x_j}U\chi$ where $\chi=\chi\left(p\right)\in C_c^\infty\left(\R^2\right)$ with $\int{\chi\,dp}=1$. Note that $\partial_t\partial_{x_j}U$ is continuous and of class $W^{1,1}\left(0,T;L^\infty\right)$ by $U\in V$, and that
\begin{align*}
\int_{s-d<\left| z\right|<s}{\frac{1}{\sqrt{s^2-\left| z\right|^2}}\,dz}=2\pi\sqrt{2sd-d^2}1_{s\geq d}\leq 2\pi\sqrt{T}\sqrt{d},
\end{align*}
where $1_{s\geq d}$ denotes the indicator function of the set $\left\{s\mid s\geq d\right\}$. So there only remain the $\partial_t^2$-derivatives of $E$ and $B$. By the known convergence, we can pass to the limit in \eqref{aVM} so that the Vlasov equation holds everywhere and the Maxwell equations almost everywhere. With this knowledge and the just proven fact that the second space derivatives of the fields are continuous, we conclude that also the $\partial_t^2$-derivatives are continuous.

Now the fact that all weak derivatives are continuous instantly implies that they are classical ones. Therefore the fields are of class $C^2$. Thus the characteristics
\begin{align*}
\dot{X}=\widehat{P},\,\dot{P}=\alpha\left(P\right)\left(E-\widehat{P}^\bot B\right)\left(s,X\right),\,\left(X,P\right)\left(t,t,x,p\right)=\left(x,p\right)
\end{align*}
are well defined and of class $C^2$ in $\left(t,x,p\right)$. Hence
\begin{align*}
f\left(t,x,p\right)=\mathring{f}\left(\left(X,P\right)\left(0,t,x,p\right)\right)
\end{align*}
is also of class $C^2$.

Therefore, we are able to pass to the limit in \eqref{aVM}, but actually \eqref{CVM} is to be solved: Obviously, \eqref{aVM} coincides with \eqref{CVM2nd} as long as $f$ vanishes for $\left| p\right|\geq Q$. But this property is guaranteed by Lemma \ref{Q}. Therefore $\left(f,E,B\right)$ is a solution of \eqref{CVM2nd} and hence of \eqref{CVM} by equivalence.

We collect some properties of $\left(f,E,B\right)$:
\begin{theorem}\label{prop}
There is a solution $\left(f,E,B\right)$ of \eqref{CVM} with:
\begin{enumerate}[i)]
	\item $f$, $E$, and $B$ are of class $C^2$,
	\item $f$ vanishes for $\left| p\right|\geq Q$ or $\left| x\right|\geq R+T$ (where $Q$ only depends on $T$, the initial data (their $C_b^1$-norms and $P\left(0\right)$), and $\left\|U\right\|_V$, and where $\supp_x\mathring{f}\subset B_R$),
	\item $E$, $B$ vanish for $\left| x\right|\geq\widetilde{R}+L+R+T$ if their initial data are compactly supported, i.e. $\supp\,\mathring{E}$, $\supp\,\mathring{B}\subset B_{\widetilde{R}}$,
	\item the $C_b^2$-norms of the solution are estimated by a constant only depending on $T$, the initial data (their $C_b^2$-norms and $P\left(0\right)$), $L$, and $\left\|U\right\|_V$.
\end{enumerate}
\end{theorem}
\begin{proof}
For ii) note that $\left|\dot{X}\right|\leq 1$, for iii) recall the representation formula of the fields, and iv) holds because it holds for all iterates, they converge in $C_b^1$ and their second derivatives weakly-* in $L^\infty$.
\end{proof}
\subsubsection{Uniqueness}
We prove uniqueness of the solution.
\begin{theorem}
The obtained solution $\left(f,E,B\right)$ of \eqref{CVM} is unique in $C^1\times\left(C^2\right)^2$.
\end{theorem}
\begin{proof}
The proof is standard. Consider the difference of two solutions and apply Lemmas \ref{fest} and \ref{field} to show that the difference vanishes after a Gronwall argument.
\end{proof}
Moreover, it is possible to show that the solution is unique in an even larger class. Here, the constructed solution satisfies the conditions if $\mathring{E}$ and $\mathring{B}$ are compactly supported.
\begin{theorem}\label{uniq}
A solution $\left(f,E,B\right)$ of \eqref{CVM} with the properties
\begin{enumerate}[i)]
	\item $f$, $E$, and $B$ are of class $W^{1,\infty}\cap H^1$,
	\item $\supp f\subset\left[0,T\right]\times B_r^2$ for some $r>0$,
\end{enumerate}
is unique (here, 'solution' means that \eqref{CVM} holds pointwise almost everywhere).
\end{theorem}
\begin{proof}
Let $\left(\widetilde{f},\widetilde{E},\widetilde{B}\right)$ (with the above properties) solve \eqref{CVM} too and define $\overline{f}:=\widetilde{f}-f$ and so on. Then we have the system
\begin{align*}
\partial_t\overline{f}+\widehat{p}\cdot\partial_x\overline{f}+\left(\widetilde{E}-\widehat{p}^\bot\widetilde{B}\right)\cdot\partial_p\overline{f} &= -\left(\overline{E}-\widehat{p}^\bot\overline{B}\right)\cdot\partial_pf,\\
\partial_t\overline{E}_1-\partial_{x_2}\overline{B} &= -j_{\overline{f},1},\\
\partial_t\overline{E}_2+\partial_{x_1}\overline{B} &= -j_{\overline{f},2},\\
\partial_t\overline{B}+\partial_{x_1}\overline{E}_2-\partial_{x_2}\overline{E}_1 &= 0,\\
\left(\overline{f},\overline{E},\overline{B}\right)\left(0\right) &= 0.
\end{align*}
Note that initial values make sense because of $H^1\subset H^1\left(0,T;L^2\right)\hookrightarrow C\left(0,T;L^2\right)$. We have
\begin{align*}
& \frac{1}{2}\left\|\overline{f}\left(t\right)\right\|^2_{L^2}\\
&= \int_0^t{\int{\int{\overline{f}\partial_t\overline{f}\,dp}dx}d\tau}\\
&= \int_0^t{\int{\int{\overline{f}\left(-\widehat{p}\cdot\partial_x\overline{f}-\left(\widetilde{E}-\widehat{p}^\bot\widetilde{B}\right)\cdot\partial_p\overline{f}-\left(\overline{E}-\widehat{p}^\bot\overline{B}\right)\cdot\partial_pf\right)dp}dx}d\tau}\\
&= \int_0^t{\int{\int{\left(-\frac{1}{2}\divg_x\left(\widehat{p}\overline{f}^2\right)-\frac{1}{2}\divg_p\left(\left(\widetilde{E}-\widehat{p}^\bot\widetilde{B}\right)\overline{f}^2\right)-\overline{f}\left(\overline{E}-\widehat{p}^\bot\overline{B}\right)\cdot\partial_pf\right)}}}\\
& \phantom{\int_0^t\int{\int{}}}dpdxd\tau\\
&= -\int_0^t{\int{\int{\overline{f}\left(\overline{E}-\widehat{p}^\bot\overline{B}\right)\cdot\partial_pf\,dp}dx}d\tau}\\
&\leq \left\|f\right\|_{W^{1,\infty}}\int_0^t{\left\|\overline{f}\left(\tau\right)\right\|_{L^2}\left(\left\|\overline{E}\left(\tau\right)\right\|_{L^2}+\left\|\overline{B}\left(\tau\right)\right\|_{L^2}\right)d\tau},
\end{align*}
which implies
\begin{align*}
\left\|\overline{f}\left(t\right)\right\|_{L^2}\leq\left\|f\right\|_{W^{1,\infty}}\int_0^t{\left(\left\|\overline{E}\left(\tau\right)\right\|_{L^2}+\left\|\overline{B}\left(\tau\right)\right\|_{L^2}\right)d\tau}
\end{align*}
via the quadratic version of Gronwall's inequality, cf. \cite[Thm. 5]{gron}. Similarly,
\begin{align*}
\frac{1}{2}\left\|\overline{B}\left(t\right)\right\|^2_{L^2} &= \int_0^t{\int{\overline{B}\partial_t \overline{B}\,dx}d\tau} = \int_0^t{\int{\overline{B}\left(-\partial_{x_1}\overline{E}_2+\partial_{x_2}\overline{E}_1\right)dx}d\tau}\\
&= \int_0^t{\int{\left(\overline{E}_2\partial_{x_1}\overline{B}-\overline{E}_1\partial_{x_2}\overline{B}\right)dx}d\tau}\\
&= \int_0^t{\int{\left(-\overline{E}\cdot\partial_t\overline{E}-\overline{E}\cdot j_{\overline{f}}\right)dx}d\tau}.
\end{align*}
Note that in the integration by parts no surface terms appear because of $E$, $B\in H^1$. This computation leads to
\begin{align*}
& \frac{1}{2}\left(\left\|\overline{E}\left(t\right)\right\|^2_{L^2}+\left\|\overline{B}\left(t\right)\right\|^2_{L^2}\right) = \int_0^t{\int{-\overline{E}\cdot j_{\overline{f}}\,dx}d\tau}\\
&\leq \int_0^t{\left\|\overline{E}\left(\tau\right)\right\|_{L^2}\left\|j_{\overline{f}}\left(\tau\right)\right\|_{L^2}d\tau} \leq C\left(r\right)\int_0^t{\left(\left\|\overline{E}\left(\tau\right)\right\|_{L^2}+\left\|\overline{B}\left(\tau\right)\right\|_{L^2}\right)\left\|\overline{f}\left(\tau\right)\right\|_{L^2}d\tau}.
\end{align*}
Here, the last inequality holds because $\overline{f}$ vanishes as soon as $\left| p\right|>r$. Now again, the quadratic Gronwall lemma implies
\begin{align*}
\left\|\overline{E}\left(t\right)\right\|_{L^2}+\left\|\overline{B}\left(t\right)\right\|_{L^2} &\leq C\left(r\right)\int_0^t{\left\|\overline{f}\left(\tau\right)\right\|_{L^2}d\tau}\\
&\leq C\left(r,T\right)\left\|f\right\|_{W^{1,\infty}}\int_0^t{\left(\left\|\overline{E}\left(\tau\right)\right\|_{L^2}+\left\|\overline{B}\left(\tau\right)\right\|_{L^2}\right)d\tau}.
\end{align*}
This yields $\left(\overline{E},\overline{B}\right)=0$ and hence also $\overline{f}=0$.
\end{proof}

\section{The control-to-state operator}
From now on the initial data always stay fixed with $0\leq\mathring{f}\in C_c^2$ and $\mathring{E}$, $\mathring{B}\in C_c^3$, and $\divg\mathring{E}=\rho_{\mathring{f}}$. As a result of the last section we may define the control-to-state operator via
\begin{align*}
S\colon V&\rightarrow C_b^2\left(\left[0,T\right]\times\R^4\right)\times C_b^2\left(\left[0,T\right]\times\R^2;\R^2\right)\times C_b^2\left(\left[0,T\right]\times\R^2\right),\\
U&\mapsto \left(f,E,B\right).
\end{align*}
The goal is to show that $S$ is differentiable with respect to suitable norms.
\subsection{Lipschitz continuity}\label{lips}
First we show that $S$ is Lipschitz continuous; to be more precise, locally Lipschitz continuous. Let $U$, $\delta U\in V$ and denote $\left(f,E,B\right)=S\left(U\right)$, $\left(\overline{f},\overline{E},\overline{B}\right)=S\left(U+\delta U\right)$, and $\left(\widetilde{f},\widetilde{E},\widetilde{B}\right)=S\left(U+\delta U\right)-S\left(U\right)$. We arrive at the system
\begin{align*}
\partial_t\widetilde{f}+\widehat{p}\cdot\partial_x\widetilde{f}+\left(E-\widehat{p}^\bot B\right)\cdot\partial_p\widetilde{f} &= -\left(\widetilde{E}-\widehat{p}^\bot\widetilde{B}\right)\cdot\partial_p\overline{f},\\
\partial_t\widetilde{E}_1-\partial_{x_2}\widetilde{B} &= -j_{\widetilde{f},1}-\delta U_1,\\
\partial_t\widetilde{E}_2+\partial_{x_1}\widetilde{B} &= -j_{\widetilde{f},2}-\delta U_2,\\
\partial_t\widetilde{B}+\partial_{x_1}\widetilde{E}_2-\partial_{x_2}\widetilde{E}_1 &= 0,\\
\left(\widetilde{f},\widetilde{E},\widetilde{B}\right)\left(0\right) &= 0,
\end{align*}
which is equivalent to the system with second order Maxwell equations because of Lemmas \ref{equi} and \ref{LCco}.

Note that the $x$- and $p$-support of the density and the $C_b^1$-norm of the solution is controlled by a constant dependent on $T$, the initial data, $L$, and the $V$-norm of the control, see Theorem \ref{prop}. Therefore we can perform the same estimates also on the $\overline{\cdot}$-solution with a constant dependent on $T$, the initial data, $L$, and $\left\|U\right\|_V$ because, for instance, for $\left\|\delta U\right\|_V\leq 1$ we have $\left\|U+\delta U\right\|_V\leq\left\|U\right\|_V+1$. Hence we will only show the locally Lipschitz continuity of $S$.

Indeed, using again the estimates of Lemmas \ref{fest}, \ref{field}, and \ref{fieldder}, we see that
\begin{align*}
\left\|\left(\widetilde{f},\widetilde{E},\widetilde{B}\right)\right\|_{C_b^1}\leq C\left\|\delta U\right\|_V.
\end{align*}
Thus we have proved:
\begin{lemma}\label{lip}
$S\colon V\rightarrow C^1_b\left(\left[0,T\right]\times\R^4\right)\times C^1_b\left(\left[0,T\right]\times\R^2\right)^3$ is locally Lipschitz continuous.
\end{lemma}
\subsection{Solvability of a linearized system}
To show even differentiability of $S$ we will have to analyze a linearized system of the form
\begin{align}\tag{LVM}\label{LVM}\left.\begin{aligned}
\partial_t f+\widehat{p}\cdot\partial_x f+G\cdot\partial_p f &= \left(E-\widehat{p}^\bot B\right)\cdot g+a,\\
\partial_t E_1-\partial_{x_2}B &= -j_{f,1}+h_1,\\
\partial_t E_2+\partial_{x_1}B &= -j_{f,2}+h_2,\\
\partial_t B+\partial_{x_1}E_2-\partial_{x_2}E_1 &= 0,\\
\left(f,E,B\right)\left(0\right) &= 0
\end{aligned}\right\}\end{align}
with already given functions $a\in L^1\left(0,T;L^2\right)$, $G\in C_b^2$ with $\divg_pG=0$, $g\in C_b^1$ with $g=\partial_p\widetilde{g}$ for some $\widetilde{g}\in C_b^2$ and $g\left(t,x,p\right)=0$ for $\left| x\right|\geq r$ or $\left| p\right|\geq r$ for some $r>0$, and $h\in V$. We call $\left(f,E,B\right)$ a solution of \eqref{LVM} if $f$, $E$, and $B$ are of class $C\cap H^1$, the equalities hold pointwise almost everywhere, and $f$ vanishes for $\left| p\right|\geq R$ for some $R>0$.

A crucial estimate is the following:
\begin{lemma}\label{L2}
Let $\left(f,E,B\right)$ be a solution of \eqref{LVM}. Then
\begin{align*}
\left\|f\left(t\right)\right\|_{L^2}+\left\|E\left(t\right)\right\|_{L^2}+\left\|B\left(t\right)\right\|_{L^2}\leq C\left(R,\left\|g\right\|_\infty,T\right)\int_0^t{\left(\left\|a\left(\tau\right)\right\|_{L^2}+\left\|h\left(\tau\right)\right\|_{L^2}\right)d\tau}.
\end{align*}
\end{lemma}
\begin{proof}
The proof is very similar to that of Theorem \ref{uniq} and is omitted.
\end{proof}
We approximate $G$, $\widetilde{g}$, and $h$ with smooth functions $G_k$, $\widetilde{g}_k$, and $h_k$ which are converging to $G$, $\widetilde{g}$, and $h$ in $C_b^2$ and $V$ respectively, and define $g_k:=\partial_p\widetilde{g}_k$. 

To show solvability of \eqref{LVM} for $a=0$ we proceed similarly as before. Define $f_0=E_{0,1}=E_{0,2}=B_0=0$ and solve in the $k$-th step
\begin{align*}
\partial_t f_k+\widehat{p}\cdot\partial_x f_k+G_k\cdot\partial_p f_k &= \left(E_{k-1}-\widehat{p}^\bot B_{k-1}\right)\cdot g_k,\\
f_k\left(0\right) &= 0
\end{align*}
by defining
\begin{align*}
f_k\left(t,x,p\right)=\int_0^t{\left(\left(E_{k-1}-\widehat{p}^\bot B_{k-1}\right)\cdot g_k\right)\left(X_k\left(0,t,x,p\right),P_k\left(0,t,x,p\right)\right)d\tau}
\end{align*}
with the characteristics
\begin{align*}
\dot{X}_k&=\widehat{P}_k,&X_k\left(t,t,x,p\right)&=x,\\
\dot{P}_k&=G_k\left(s,X_k,P_k\right),&P_k\left(t,t,x,p\right)&=p,
\end{align*}
and then solving
\begin{align*}
\partial_t^2 E_k-\Delta E_k &= -\partial_t j_{f_k}-\partial_t h_k-\partial_x\rho_{f_k}+\partial_x\int_0^t{\divg_xh_k\,d\tau},\\
\partial_t^2 B_k-\Delta B_k &= \partial_{x_1}j_{f_k,2}-\partial_{x_2}j_{f_k,1}+\partial_{x_1}h_{k,2}-\partial_{x_2}h_{k,1},\\
\left(E_k,B_k\right)\left(0\right) &= 0,\\
\partial_t E_k\left(0\right) &= -U_k\left(0\right),\\
\partial_t B_k\left(0\right) &= 0.
\end{align*}
All iterates are again of class $C^\infty$. Furthermore, the characteristics are independent of the solution sequence $\left(f_k,E_k,B_k\right)$. Thus we instantly have $\left|\dot{P_k}\right|\leq C$, so $\left| P_k-p\right|\leq CT$. Having a look at the formula for $f_k$ we conclude that $f_k$ vanishes as soon as
\begin{align}\label{dfp}
\left| p\right|\geq 2r+CT=:Q
\end{align}
since then the integrand vanishes as a result of
\begin{align*}
\left| P_k\left(s,t,x,p\right)\right|\geq \left| p\right|-\left| P_k-p\right|\geq 2r+CT-CT=2r.
\end{align*}
The same can be done for the $x$-coordinate starting with $\left|\dot{X_k}\right|\leq 1$; hence $f_k\left(t,x,p\right)=0$ for $\left| x\right|\geq 2r+T$. The assertions of Section \ref{est} are directly applicable. We do not have to insert some $\alpha$ because of the already known bound on the $p$-support of $f_k$. Therefore \eqref{LC} holds for the iterated system and we can thus switch between first order and second order Maxwell equations; note that $\left(E_{k-1}-\widehat{p}^\bot B_{k-1}\right)\cdot g_k=\divg_p\left(\left(E_{k-1}-\widehat{p}^\bot B_{k-1}\right)\widetilde{g}_k\right)$.

We proceed like in Section \ref{loc}: The iterates are bounded in $C_b^1$ and are Cauchy with respect to the $C_b$-norm. However, after that there appears a difference:  Unfortunately, we can not show the Cauchy property with respect to the $C_b^1$-norm. For this we would first have to bound second derivatives of $f_k$ which would require control of second derivatives of $g_k$. This, on the other hand, would require a smoother $g$. But for the later application we will not have more regularity of $g$ than $C_b^1$.

Thus we have to proceed differently: Since $f_k$, $E_k$, and $B_k$ are bounded in the $C_b^1$-norm, their first derivatives converge, after extracting a suitable subsequence, to the respective derivatives of $f$, $E$, and $B$ in $L^\infty$ in the weak-*-sense. Because of
\begin{align*}
& \left|\int_0^T{\int{\int{\left(G_k\cdot\partial_pf_k\varphi-G\cdot\partial_pf\varphi\right)dp}dx}d\tau}\right|\\
&\leq \int_0^T{\int{\int{\left| G_k-G\right|\left|\partial_pf_k\right|\left|\varphi\right| dp}dx}d\tau}+\left|\int_0^T{\int{\int{G\left(\partial_pf_k-\partial_pf\right)\varphi\,dp}dx}d\tau}\right|\\
&\leq C\left\|G_k-G\right\|_\infty\left\|\varphi\right\|_{L^1}+\left|\int_0^T{\int{\int{G\left(\partial_pf_k-\partial_pf\right)\varphi\,dp}dx}d\tau}\right|\rightarrow 0
\end{align*}
for $k\rightarrow\infty$ for any test function $\varphi$, $\left(f,E,B\right)$ satisfies \eqref{LVM} pointwise almost everywhere; the other terms are obviously easier to handle. Altogether we have found a solution of \eqref{LVM} of class $C\cap W^{1,\infty}$. Furthermore it is also of class $H^1$ because all sequence elements have compact support with respect to $x$, $p$ or $x$ respectively uniformly in $t$ and $k$; for the fields recall the representation formula.

For uniqueness, let $\left(f_1,E_1,B_1\right)$ be a solution of \eqref{LVM} too and define $f_2:=f-f_1$ and so on which yields
\begin{align*}
\partial_t f_2+\widehat{p}\cdot\partial_x f_2+G\cdot\partial_p f_2 &= \left(E_2-\widehat{p}^\bot B_2\right)\cdot g,\\
\partial_t E_{2,1}-\partial_{x_2}B_2 &= -j_{f_2,1},\\
\partial_t E_{2,2}+\partial_{x_1}B_2 &= -j_{f_2,2},\\
\partial_t B_2+\partial_{x_1}E_{2,2}-\partial_{x_2}E_{2,1} &= 0,\\
\left(f_2,E_2,B_2\right)\left(0\right) &= 0.
\end{align*}
Applying Lemma \ref{L2} this instantly implies that $f_2$, $E_2$, and $B_2$ vanish.

\subsection{Differentiability}
We want to study the differentiability of $S\colon V\rightarrow C\left(0,T;L^2\left(\R^4\right)\right)\times C\left(0,T;L^2\left(\R^2\right)\right)^3$. Let $U\in V$ and let $\delta U\in V$ be some perturbation. In the following denote $\left(f,E,B\right)=S\left(U\right)$ and $\left(\overline{f},\overline{E},\overline{B}\right)=S\left(U+\delta U\right)$. The candidate for the linearization is $S'\left(U\right)\delta U=\left(\delta f,\delta E,\delta B\right)$ where the right hand side satisfies
\begin{align*}
\partial_t\delta f+\widehat{p}\cdot\partial_x\delta f+\left(E-\widehat{p}^\bot B\right)\cdot\partial_p\delta f &= -\left(\delta E-\widehat{p}^\bot\delta B\right)\cdot\partial_pf,\\
\partial_t\delta E_1-\partial_{x_2}\delta B &= -j_{\delta f,1}-\delta U_1,\\
\partial_t\delta E_2+\partial_{x_1}\delta B &= -j_{\delta f,2}-\delta U_2,\\
\partial_t\delta B+\partial_{x_1}\delta E_2-\partial_{x_2}\delta E_1 &= 0,\\
\left(\delta f,\delta E,\delta B\right)\left(0\right) &= 0.
\end{align*}
Indeed, this system can be solved because of $G:=E-\widehat{p}^\bot B\in C_b^2$ (note that $\divg_pG=0$), $g:=-\partial_pf\in C_b^1$, and $h:=\delta U\in V$. First we note that $S'\left(U\right)$ is linear and that by Lemma \ref{L2}
\begin{align}\label{delta}
\left\|\left(\delta f,\delta E,\delta B\right)\right\|_{C\left(0,T;L^2\right)} \leq C\int_0^T{\left\|\delta U\left(t\right)\right\|_{L^2}dt}\leq C\left\|\delta U\right\|_V
\end{align}
which says that $S'\left(U\right)$ is bounded. The last inequality holds because of $\supp\,\delta U\left(t\right)\subset B_L$.

The next step is to show that $S\left(U+\delta U\right)-S\left(U\right)-S'\left(U\right)\delta U$ is 'small'. Defining $\widetilde{f}:=\overline{f}-f-\delta f$ and so on and subtracting the respective equations yield
\begin{align*}
\partial_t\widetilde{f}+\widehat{p}\cdot\partial_x\widetilde{f}+\left(E-\widehat{p}^\bot B\right)\cdot\partial_p\widetilde{f} &= -\left(\widetilde{E}-\widehat{p}^\bot\widetilde{B}\right)\cdot\partial_pf\\
&\phan -\left(\overline{E}-E-\widehat{p}^\bot\left(\overline{B}-B\right)\right)\cdot\partial_p\left(\overline{f}-f\right),\\
\partial_t\widetilde{E}_1-\partial_{x_2}\widetilde{B} &= -j_{\widetilde{f},1},\\
\partial_t\widetilde{E}_2+\partial_{x_1}\widetilde{B} &= -j_{\widetilde{f},2},\\
\partial_t\widetilde{B}+\partial_{x_1}\widetilde{E}_2-\partial_{x_2}\widetilde{E}_1 &= 0,\\
\left(\widetilde{f},\widetilde{E},\widetilde{B}\right)\left(0\right) &= 0.
\end{align*}
Applying Lemma \ref{L2} we conclude
\begin{align*}
\left\|\left(\widetilde{f},\widetilde{E},\widetilde{B}\right)\right\|_{C\left(0,T;L^2\right)}\leq C\int_0^T{\left\|a\left(t\right)\right\|_{L^2}dt}
\end{align*}
where
\begin{align*}
a:=-\left(\overline{E}-E-\widehat{p}^\bot\left(\overline{B}-B\right)\right)\cdot\partial_p\left(\overline{f}-f\right).
\end{align*}
Here we have to exploit the Lipschitz property of $S$. Lemma \ref{lip} yields
\begin{align*}
\left\|a\left(t\right)\right\|_{L^2}\leq C\left(\left\|\overline{E}-E\right\|_\infty+\left\|\overline{B}-B\right\|_\infty\right)\left\|\overline{f}-f\right\|_{C_b^1}\leq C\left\|\delta U\right\|_V^2.
\end{align*}
Note that for the first inequality the fact was used that $\overline{f}$ and $f$ have compact support in $x$ and $p$ uniformly in $t$ and independent of $\left\|\delta U\right\|_V$ for, for instance, $\left\|\delta U\right\|_V\leq 1$ (recall Theorem \ref{prop} and the reasoning in Section \ref{lips}).

Finally we arrive at
\begin{align}\label{widet}
\left\|\left(\widetilde{f},\widetilde{E},\widetilde{B}\right)\right\|_{C\left(0,T;L^2\right)}\leq C\left\|\delta U\right\|_V^2
\end{align}
which proves part of i) of the following theorem:
\begin{theorem}\label{diff} The following maps are continuously Fr\'echet-differentiable with locally Lipschitz derivative:
\begin{enumerate}[i)]
	\item $S\colon V\rightarrow W:=C\left(0,T;L^2\left(\R^4\right)\right)\times C\left(0,T;L^2\left(\R^2\right)\right)^3$,
	\item $\Phi:=\rho\circ S_1\colon V\rightarrow C\left(0,T;L^2\left(\R^2\right)\right)$, $U\mapsto\rho_f$,
	\item $\overline{\Phi}:=\rho\circ S_1\colon V\rightarrow C\left(0,T;L^1\left(\R^2\right)\right)$, $U\mapsto\rho_f$.
\end{enumerate}
\end{theorem}
\begin{proof}
For part ii) define
\begin{align}\label{phi'}
\Phi'\left(U\right)\delta U:=\rho_{\delta f}.
\end{align}
Now it is crucial to bound the $p$-support of $\overline{f}$, $f$, and $\delta f$ by a constant $C>0$ only depending on $T$, the initial data, $L$, and $\left\|U\right\|_V$. We first consider $\delta f$. The control of the $p$-support in \eqref{dfp} holds for all iterates and hence for $\delta f$. The constant there only depends on $T$, $\left\|G\right\|_\infty=\left\|E-\widehat{p}^\bot B\right\|_\infty$, the $p$-support of $\partial_p f$, and $L$. Because of Theorem \ref{prop} the absolute values of the fields $E$ and $B$ and the $p$-support of $f$ are controlled by some constant only depending on $T$, the initial data, $L$, and $\left\|U\right\|_V$. Hence we have together with \eqref{delta}
\begin{align*}
\left\|\rho_{\delta f}\left(t\right)\right\|_{L^2}=\left(\int{\left|\int{\delta f\,dp}\right|^2 dx}\right)^{\frac{1}{2}}\leq C\left(\int{\int{\left|\delta f\right|^2dp}dx}\right)^{\frac{1}{2}}\leq C\left\|\delta U\right\|_V
\end{align*}
which implies that $\Phi'\left(U\right)$ is bounded. Furthermore the $p$-supports of $\overline{f}$ and $f$ only depend on $T$, the initial data, $L$, and $\left\|U\right\|_V$ (for again $\left\|\delta U\right\|_V\leq 1$ for example). Hence the same assertion holds for $\widetilde{f}=\overline{f}-f-\delta f$ and therefore with \eqref{widet}
\begin{align*}
\left\|\rho_{\widetilde f}\left(t\right)\right\|_{L^2}=\left(\int{\left|\int{\widetilde f\,dp}\right|^2 dx}\right)^{\frac{1}{2}}\leq C\left(\int{\int{\left|\widetilde f\right|^2dp}dx}\right)^{\frac{1}{2}}\leq C\left\|\delta U\right\|_V^2.
\end{align*}
Together with the equality
\begin{align*}
\Phi\left(U+\delta U\right)-\Phi\left(U\right)-\Phi'\left(U\right)\delta U=\rho_{\overline{f}}-\rho_f-\rho_{\delta f}=\rho_{\widetilde{f}}
\end{align*}
this instantly yields that $\Phi'\left(U\right)$ is indeed the Fr\'echet-derivative of $\Phi$ in $U$. Part iii) is an instant consequence of ii) and the support assertions discussed above. The derivative of $\overline{\Phi}$ is given by \eqref{phi'} as before.

To show continuity of $S'$, let $\delta V\in V$ with $\left\|\delta V\right\|_V\leq 1$. We have to investigate
\begin{align*}
\left(\check{f},\check{E},\check{B}\right):=\left(f^1,E^1,B^1\right)-\left(f^0,E^0,B^0\right):=S'\left(U+\delta U\right)\delta V-S'\left(U\right)\delta V.
\end{align*}
Applying the previously given formula for $S'$ we arrive at
\begin{align*}
\partial_t\check{f}+\widehat{p}\cdot\partial_x\check{f}+\left(\overline{E}-\widehat{p}^\bot\overline{B}\right)\cdot\partial_p\check{f} &= -\left(\check{E}-\widehat{p}^\bot\check{B}\right)\cdot\partial_p\overline{f}-\left(E^0-\widehat{p}^\bot B^0\right)\cdot\partial_p\left(\overline{f}-f\right)\\
&\phan -\left(\overline{E}-E-\widehat{p}^\bot\left(\overline{B}-B\right)\right)\cdot\partial_pf^0,\\
\partial_t\check{E}_1-\partial_{x_2}\check{B} &= -j_{\check{f},1},\\
\partial_t\check{E}_2+\partial_{x_1}\check{B} &= -j_{\check{f},2},\\
\partial_t\check{B}+\partial_{x_1}\check{E}_2-\partial_{x_2}\check{E}_1 &= 0,\\
\left(\check{f},\check{E},\check{B}\right)\left(0\right) &= 0.
\end{align*}
We know that the $p$-support of $f^0$ and the absolute values of $E^0$ and $B^0$ are controlled by a constant only depending on $T$, the initial data, $L$, $\left\|U\right\|_V$, and $\left\|\delta V\right\|_V$ (the latter can be neglected, of course). The dependence on some terms in $f$, $E$, and $B$ can be eliminated like in the beginning of this proof. Hence, proceeding as before and using Lemma \ref{L2} and the locally Lipschitz continuity of $S$, we conclude
\begin{align*}
\left\|\left(\check{f},\check{E},\check{B}\right)\right\|_W\leq C\left\|\delta U\right\|_V
\end{align*}
where $C$ only depends on $T$, the initial data, $L$, and $\left\|U\right\|_V$. This leads to
\begin{align*}
\left\|S'\left(U+\delta U\right)-S'\left(U\right)\right\|_{L\left(V,W\right)}\leq C\left\|\delta U\right\|_V
\end{align*}
which says that $S'$ is even locally Lipschitz continuous.

Using the assertions for the $p$-support of $f^0$ and $f^1$ (controlled by a constant only depending on $T$, the initial data, $L$, and $\left\|U\right\|_V$ if $\left\|\delta U\right\|_V\leq 1$) we conclude
\begin{align*}
\left\|\rho_{\check{f}}\right\|_{C\left(0,T;L^2\right)},\left\|\rho_{\check{f}}\right\|_{C\left(0,T;L^1\right)}\leq C\left\|\check{f}\right\|_{C\left(0,T;L^2\right)}\leq C\left\|\delta U\right\|_V
\end{align*}
as before. This implies that $\Phi'$ and $\overline{\Phi}'$ are locally Lipschitz continuous.
\end{proof}

\section{Optimal control problem}
Now we consider some optimal control problems. We want to minimize some objective function that depends on the external control $U$ and the state $\left(f,E,B\right)$. The control and the state are coupled via \eqref{CVM} so that \eqref{CVM} appears as a constraint.

We first give thought to a problem with general controls and a general objective function. Then we proceed with optimizing problems where the objective function is explicitly given and where the control set is restricted to such controls that are realizable in applications concerning the control of a plasma.
\subsection{General problem}
\subsubsection{Control space}
Until now we have worked with the control space
\begin{align*}
V=\left\{U\in W^{2,1}\left(0,T;C_b^4\left(\R^2;\R^2\right)\right)\mid U\left(t,x\right)=0\text{ for }\left| x\right|\geq L\right\}.
\end{align*}
To apply standard optimization techniques it is necessary that the control space is reflexive. Hence we choose
\begin{align*}
\mathcal{U}:=\left\{U\in H^2\left(0,T;W^{5,\gamma}\left(\R^2;\R^2\right)\right)\mid U\left(t,x\right)=0\text{ for }\left| x\right|\geq L\right\},
\end{align*}
where $\gamma>2$ is fixed, equipped with the $H^2\left(0,T;W^{5,\gamma}\right)$-norm. By Sobolev's embedding theorems, $\mathcal{U}$ is continuously embedded in $V$.

In accordance with Theorems \ref{prop} and \ref{diff}, we have already proved that there is a continuously differentiable control-to-state operator
\begin{align*}
S\colon V&\rightarrow \left(C_b^2\left(\left[0,T\right]\times\R^4\right)\times C_b^2\left(\left[0,T\right]\times\R^2;\R^2\right)\times C_b^2\left(\left[0,T\right]\times\R^2\right),\left\|\cdot\right\|_{C\left(0,T;L^2\right)}\right),\\
U&\mapsto \left(f,E,B\right),
\end{align*}
such that \eqref{CVM} holds for $\left(f,E,B\right)$ and control $U$. Furthermore, the map $U\mapsto\rho_f$ is continuously differentiable with respect to the $C\left(0,T;L^2\right)$- and $C\left(0,T;L^1\right)$-norm in the image space. Moreover, the $C_b^2$-norm and the $x$- and $p$-support of $\left(f,E,B\right)$ are controlled by a constant only depending on $T$, $L$, the initial data, and $\left\|U\right\|_V$.

By $\mathcal{U}\hookrightarrow V$, these assertions also hold with $\mathcal{U}$ instead of $V$.
\subsubsection{Existence of minimizers}\label{min}
We consider the general problem
\begin{align}\tag{GP}\label{GP}\left.\begin{aligned}
\min_{\substack{\left(f,E,B\right)\in\left(C^2\cap H^1\right)^3,U\in\mathcal{U}}}\,&\phi\left(f,E,B,U\right)\\
\text{s.t.}\,&\left(f,E,B\right)=S\left(U\right).
\end{aligned}\right\}\end{align}
We have to specify some assumptions on $\phi$:
\begin{condition}\makeatletter\hyper@anchor{\@currentHref}\makeatother\label{phico}
\begin{enumerate}[i)]
	\item $\phi\colon\left(C^2\cap H^1\right)^3\times\mathcal{U}\rightarrow\R\cup\left\{\infty\right\}$ and $\phi\not\equiv\infty$,
	\item $\phi$ is coercive in $U\in\mathcal{U}$, i.e. in general: Let $X$, $Y$ be normed spaces; $\psi\colon X\times Y\rightarrow\R$ is said to be coercive in $y\in Y$ iff for all sequences $\left(y_k\right)\subset Y$ with $\left\|y_k\right\|_Y\rightarrow\infty$, $k\rightarrow\infty$, then also $\psi\left(x_k,y_k\right)\rightarrow\infty$, $k\rightarrow\infty$, for any sequence $\left(x_k\right)\subset X$,
	\item $\phi$ is weakly lower semicontinuous, i.e.: if $\left(f_k,E_k,B_k\right)\rightharpoonup\left(f,E,B\right)$ in $H^1$ and $U_k\rightharpoonup U$ in $\mathcal{U}$, then $\phi\left(f,E,B,U\right)\leq\liminf_{\substack{k\rightarrow\infty}}\phi\left(f_k,E_k,B_k,U_k\right)$.
\end{enumerate}
\end{condition}
These assumptions allow us to prove existence of a (not necessarily unique) minimizer. We will first prove a lemma that will be useful later:
\begin{lemma}\label{minlma}
Let $\left(U_k\right)\subset V$ be bounded and $\left(f_k,E_k,B_k\right)=S\left(U_k\right)$. Then, after extracting a suitable subsequence, it holds that:
\begin{enumerate}[i)]
	\item The sequences $\left(f_k\right)$, $\left(E_k\right)$, and $\left(B_k\right)$ converge weakly in $H^1$, weakly-* in $W^{1,\infty}$, and strongly in $L^2$ to some $f$, $E$, and $B$.
	\item There is $r>0$ such that $f$, $E$, $B$, and, for all $k\in\mathbb{N}$, $f_k$, $E_k$, and $B_k$ vanish if $\left| x\right|\geq r$ or $\left| p\right|\geq r$.
	\item If additionally $U_k\rightarrow U$ in the sense of distributions for some $U\in V$ for $k\rightarrow\infty$, then $\left(f,E,B\right)=S\left(U\right)$ and $f$, $E$, and $B$ are of class $C_b^2$.
\end{enumerate}
\end{lemma}
\begin{proof}
By Theorem \ref{prop}, on the one hand, $\left(f_k,E_k,B_k\right)$ is bounded in the $C_b^1$-norm. On the other hand, $f_k$ vanishes as soon as $\left| p\right|$ is large enough uniformly in $k$. Moreover, $f_k$, $E_k$, and $B_k$ vanish as soon as $\left| x\right|$ is large enough. Hence $\left(f_k,E_k,B_k\right)$ is also bounded in $H^1$ and in $H^1\left(0,T;L^2\right)$. Together with the boundedness in $C_b^1$, $\left(f_k,E_k,B_k\right)$ converge, after extracting a suitable subsequence, to some $\left(f,E,B\right)$, namely weakly in $H^1$, and weakly-* in $W^{1,\infty}$. This proves ii) and part of i).

For the remaining part of i) (strong convergence in $L^2$) we have to exploit some compactness. This compactness is guaranteed by the theorem of Rellich-Kondrachov. By the reasoning above, $\left(f_k,E_k,B_k\right)$ are bounded in $H^1$ and in fact, only a bounded subset of the $x$- and $p$-space matters. Hence (a subsequence of) $\left(f_k,E_k,B_k\right)$ converges strongly in $L^2$ to the limit $\left(f,E,B\right)$.

For iii), we have to pass to the limit in \eqref{CVM}. First, the initial conditions are preserved in the limit since $H^1\hookrightarrow H^1\left(0,T;L^2\right)\hookrightarrow C\left(0,T;L^2\right)$. Furthermore the Vlasov and Maxwell equations hold pointwise almost everywhere for the limit functions: The only difficult part is the nonlinear term in the Vlasov equation. To handle this, we have to make use of the strong convergence in $L^2$ obtained above. We find for each $\varphi\in C_c^\infty\left(\left]0,T\right[\times\R^4\right)$ that
\begin{align*}
& \left|\int_0^T{\int{\int{\left(\left(E_k-\widehat{p}^\bot B_k\right)\cdot\partial_pf_k-\left(E-\widehat{p}^\bot B\right)\cdot\partial_pf\right)\varphi\,dp}dx}dt}\right|\\
&\leq \left|\int_0^T{\int{\int{\left(E-\widehat{p}^\bot B\right)\cdot\left(\partial_pf_k-\partial_pf\right)\varphi\,dp}dx}dt}\right|\\
&\phan +\left\|\partial_pf_k\right\|_\infty\int_0^T{\int{\int{\left(\left| E_k-E\right|+\left| B_k-B\right|\right)\left|\varphi\right| dp}dx}dt}.
\end{align*}
Both terms converge to $0$ for $k\rightarrow\infty$ since $f_k\rightharpoonup f$ in $H^1$, $E_k\rightarrow E$, $B_k\rightarrow B$ in $L^2$, and $f_k$ is bounded in $C_b^1$. Therefore, altogether, \eqref{CVM} holds pointwise almost everywhere. Now we can apply Theorem \ref{uniq} to conclude $\left(f,E,B\right)$ equals $S\left(U\right)$ and is hence of class $C_b^2$.
\end{proof}
\begin{theorem}
Let $\phi$ satisfy Condition \ref{phico}. Then there is a minimizer of \eqref{GP}.
\end{theorem}
\begin{proof}
We consider a minimizing sequence $\left(f_k,E_k,B_k,U_k\right)$ with $\left(f_k,E_k,B_k\right)=S\left(U_k\right)$ and
\begin{align*}
\lim_{\substack{k\rightarrow\infty}}\phi\left(f_k,E_k,B_k,U_k\right)=m:=\inf_{\substack{U\in\mathcal{U},\left(f,E,B\right)=S\left(U\right)}}\phi\left(f,E,B,U\right)\in\R\cup\left\{-\infty\right\}.
\end{align*}
By coercivity in $U$, cf. Condition \ref{phico} ii), $\left(U_k\right)$ is bounded in $\mathcal{U}$ and therefore in $V$. Hence we may extract a weakly convergent subsequence (also denoted by $U_k$) since $H^2\left(0,T;W^{5,\gamma}\right)$ is reflexive. The weak limit $U$ is the candidate for being an optimal control. Of course, by weak convergence, $U$ vanishes for $\left| x\right|\geq L$; hence $U\in\mathcal{U}$. Because of $\mathcal{U}\hookrightarrow L^1$ we also get $U_k\rightharpoonup U$ in $L^1$ and hence $U_k\rightarrow U$ in the sense of distributions. Lemma \ref{minlma} yields $\left(f_k,E_k,B_k\right)\rightharpoonup\left(f,E,B\right)$ in $H^1$ (after extracting a suitable subsequence) and $\left(f,E,B\right)=S\left(U\right)$. Together with the weak lower semicontinuity of $\phi$, see Condition \ref{phico} iii), we instantly get $\phi\left(f,E,B,U\right)=m$ which proves optimality.
\end{proof}
In order to be able of examining some problem that is somehow application-oriented, we first have to think about possible problems concerning the conditions on the objective function $\phi$. Especially the coercivity in $U$ will make some trouble since the $\mathcal{U}$-norm is pretty strong. One can try to guarantee these conditions in various ways, for example if $\phi\left(f,E,B,U\right)=\psi\left(f,E,B\right)+\left\|U\right\|_{\mathcal{U}}^2$; the objective function contains some cost term of the control in the full $\mathcal{U}$-norm. But typically in applications, such a strong cost term makes no sense. Furthermore, first order optimality conditions would contain a differential equation of very high order, which is hard to solve.

On the other hand, we can not simply use a less regular control space. Firstly, we need $\mathcal{U}\hookrightarrow V$ to ensure that the control-to-state operator is differentiable; this will be useful later. Secondly, $\mathcal{U}$ needs to be reflexive to extract (in some sense) converging subsequences from a minimizing sequence. Here we should remark that we also could demand $W^{2,p}$-regularity in time for $p>1$ instead of $H^2$-regularity which would allow more controls if $1<p<2$. However, working in a $H^2$-setting (at least in time) is more convenient.
\subsection{An optimization problem with realizable external currents}
\subsubsection{Motivation}
As the previous considerations suggest, it would be nice if we somehow eliminated the variability of the control with respect to the space coordinate. This can be achieved by only considering controls of the form
\begin{align*}
U\left(t,x\right)=\sum_{j=1}^N{u_j\left(t\right)z_j\left(x\right)}
\end{align*}
where the functions $0\not\equiv z_j\in C_b^6\left(\R^2;\R^2\right)$ with $z_j$ vanishing for $\left| x\right|\geq r_j>0$ are fixed and we only vary the functions $u_j\in H^2\left(\left[0,T\right]\right)$.

From a physical point of view, this model describes an ensemble of $N$ coils with 'size' $r_j$, that stay fixed in time. Obviously, $U$ is an element of $V$ if we set $L=\max\left\{r_j\mid j=1,\dots N\right\}$. Each coil generates a current $z_j$ at full capacity that is tangential to the plane and that extends infinitely in the third space dimension. We control the system by turning these coils on whereby the capacity $u_j$ is suitably adjusted as a function of time. Hence we will have to consider an additional constraint $\left| u_j\right|\leq 1$. Physically, the consideration only of controls of the above form is no substantial restriction at all because only such control fields are realizable in applications.

A similar approach was done by P. Knopf and the author \cite{kw}.
\subsubsection{Formulation}
The problem to be considered is the following:
\begin{align}\tag{P}\label{P1}\left.\begin{aligned}
\min_{\left(f,E,B\right)\in\left(C^2\cap H^1\right)^3,u\in H^2\left(\left[0,T\right]\right)^N}\,&\frac{1}{2}\left\|\rho_f-\rho_d\right\|_{L^2\left(\left[0,T\right]\times\R^2\right)}^2+\frac{\beta}{2}\sum_{j=1}^N{c_j\left(\left\|u_j\right\|_{L^2\left(\left[0,T\right]\right)}^2\right.}\\
&\left.+\beta_1\left\|\partial_tu_j\right\|_{L^2\left(\left[0,T\right]\right)}^2+\beta_2\left\|\partial_t^2u_j\right\|_{L^2\left(\left[0,T\right]\right)}^2\right)\\
\text{s.t.}\,&\left(f,E,B\right)=S\left(\sum_{j=1}^N{u_jz_j}\right),\left| u_j\right|\leq 1
\end{aligned}\right\}\end{align}
where $c_j:=\left\|z_j\right\|_{L^2\left(\R^2;\R^2\right)}^2$. We give some comments on the objective function:
\begin{itemize}
	\item The charge density shall be as close as possible to some given desired density $\rho_d=\rho_d\left(t,x\right)\in L^2\left(\left[0,T\right]\times\R^2\right)$. One could consider the $L^2$-norm of some $f-f_d$ instead but the space coordinates of the particles are of actual interest rather than their momenta.
	\item Furthermore, the cost term containing the control shall be as small as possible. We have to use the full $H^2$-norm (an equivalent norm, to be more precise) of the $u_j$ in the regularization term so that the objective function is coercive in $u\in H^2\left(\left[0,T\right]\right)^N$. However, the $L^2$-norms of the $u_j$ itself are more interesting than the ones of their derivatives. Hence it is suitable to choose $0<\beta_1,\beta_2\ll 1$.
	\item The parameter $\beta>0$ indicates which of the two aims mentioned above shall rather be achieved.
\end{itemize}
\subsubsection{Existence of minimizers}
Section \ref{min} is useful for showing existence of minimizers of \eqref{P1}.
\begin{theorem}
There is a minimizer of \eqref{P1}.
\end{theorem}
\begin{proof}
The objective function, abbreviated by $\phi=\phi\left(f,E,B,u\right)=\phi_1\left(f\right)+\phi_2\left(u\right)$ (let $\phi_1$ be the term with $\rho_f-\rho_d$ and $\phi_2$ the remaining sum), is coercive in $u\in H^2\left(\left[0,T\right]\right)^N$ because of
\begin{align*}
\phi\left(f,E,B,u\right)\geq\frac{\beta}{2}\min\left\{1,\beta_1,\beta_2\right\}\min\left\{c_j\mid j=1,\dots,N\right\}\left\|u\right\|_{\left(H^2\right)^N}^2,
\end{align*}
where $\left\|u\right\|_{\left(H^2\right)^N}^2=\sum_{j=1}^N{\left\|u_j\right\|^2_{H^2\left(\left[0,T\right]\right)}}$. Hence, considering a minimizing sequence $\left(f_k,E_k,B_k,u^k\right)$ (we use upper indices for $u^k$ to avoid confusion with the components) with $\left(f_k,E_k,B_k\right)=S\left(\sum_{j=1}^N{u_j^kz_j}\right)$ and $\left| u_j^k\right|\leq 1$, we conclude that $\left(u^k\right)$ is bounded in $\left(H^2\right)^N$; hence $u^k\rightharpoonup u$ in $\left(H^2\right)^N$ for some $u\in\left(H^2\right)^N$ for $k\rightarrow\infty$, possibly after extracting a suitable subsequence. The constraint $\left| u_j\right|\leq1$ is obviously preserved by weak convergence. Furthermore, the sequence $\left(U_k\right):=\left(\sum_{j=1}^N{u_j^kz_j}\right)$ is bounded in $V$ because of $H^2\left(\left[0,T\right]\right)\hookrightarrow W^{2,1}\left(\left[0,T\right]\right)$.

Clearly, $U_k\to U:=\sum_{j=1}^N{u_jz_j}$ in the sense of distributions by $u_j^k\rightharpoonup u_j$ in $H^2$. Therefore, Lemma \ref{minlma} is applicable and delivers some $f$, $E$, and $B$ so that \eqref{CVM} is preserved in the limit. The remaining part is to show that $U$ is indeed an optimal control. Firstly, $u^k\rightharpoonup u$ in $\left(H^2\right)^N$ instantly implies $\phi_2\left(u\right)\leq\liminf_{\substack{k\rightarrow\infty}}\phi_2\left(u^k\right)$. Secondly, by Lemma \ref{minlma}, all $f_k$ and $f$ have compact support with respect to $p$ uniformly in $k$, and $f_k\rightarrow f$ in $L^2$. These properties yield $\rho_{f_k}\rightarrow\rho_f$ in $L^2$ by H\"older's inequality and therefore $\phi_1\left(f\right)=\lim_{\substack{k\rightarrow\infty}}\phi_1\left(f_k\right)$. This finally proves the desired optimality.
\end{proof}
\subsubsection{Differentiability of the objective function}
Next we study the differentiability of the objective function.
\begin{theorem}
\begin{enumerate}[i)]
	\item The solution map
	\begin{align*}
	\Xi\colon\left(H^2\left(\left[0,T\right]\right)\right)^N&\rightarrow C\left(0,T;L^2\left(\R^4\right)\right)\times C\left(0,T;L^2\left(\R^2\right)\right)^3,\\
	u&\mapsto \left(f,E,B\right)=S\left(\sum_{j=1}^N{u_jz_j}\right)
	\end{align*}
	is continuously Fr\'echet-differentiable and $\Xi'\left(u\right)\delta u=\left(\delta f,\delta E,\delta B\right)$ satisfies
	\begin{align*}
	\partial_t\delta f+\widehat{p}\cdot\partial_x\delta f+\left(E-\widehat{p}^\bot B\right)\cdot\partial_p\delta f &= -\left(\delta E-\widehat{p}^\bot\delta B\right)\cdot\partial_pf,\\
	\partial_t\delta E_1-\partial_{x_2}\delta B &= -j_{\delta f,1}-\delta U_1,\\
	\partial_t\delta E_2+\partial_{x_1}\delta B &= -j_{\delta f,2}-\delta U_2,\\
	\partial_t\delta B+\partial_{x_1}\delta E_2-\partial_{x_2}\delta E_1 &= 0,\\
	\left(\delta f,\delta E,\delta B\right)\left(0\right) &= 0
	\end{align*}
	where $\delta U=\sum_{j=1}^N{\delta u_jz_j}$.
	\item The maps
	\begin{align*}
	\Psi\colon\left(H^2\left(\left[0,T\right]\right)\right)^N&\rightarrow C\left(0,T;L^2\left(\R^2\right)\right),\\
	u&\mapsto \rho_f
	\end{align*}
	and
	\begin{align*}
	\overline{\Psi}\colon\left(H^2\left(\left[0,T\right]\right)\right)^N&\rightarrow C\left(0,T;L^1\left(\R^2\right)\right),\\
	u&\mapsto \rho_f
	\end{align*}
	are continuously Fr\'echet-differentiable and $\Psi'\left(u\right)\delta u=\rho_{\delta f}$ with $\delta f$ from above.
	\item The objective function
	\begin{align*}
	\overline{\phi}&\colon\left(H^2\left(\left[0,T\right]\right)\right)^N\rightarrow \R,\\
	u&\mapsto \frac{1}{2}\left\|\rho_f-\rho_d\right\|_{L^2}^2+\frac{\beta}{2}\sum_{j=1}^N{c_j\left(\left\|u_j\right\|_{L^2}^2+\beta_1\left\|\partial_tu_j\right\|_{L^2}^2+\beta_2\left\|\partial_t^2u_j\right\|_{L^2}^2\right)}
	\end{align*}
	is continuously Fr\'echet-differentiable and
	\begin{align*}
	\overline{\phi}'\left(u\right)\delta u&=\left<\rho_f-\rho_d,\rho_{\delta f}\right>_{L^2}+\beta\sum_{j=1}^N{c_j\left(\left<u_j,\delta u_j\right>_{L^2}\right.}\\
	&\phan+\left.\beta_1\left<\partial_tu_j,\partial_t\delta u_j\right>_{L^2}+\beta_2\left<\partial_t^2u_j,\partial_t^2\delta u_j\right>_{L^2}\right)
	\end{align*}
	with $\delta f$ from above.
	\end{enumerate}
\end{theorem}
\begin{proof}
Clearly, $u\mapsto\sum_{j=1}^N{u_jz_j}$ is differentiable by linearity and boundedness. Hence all assertions follow immediately by Theorem \ref{diff} and the chain rule.
\end{proof}
\subsubsection{Optimality conditions}
Now we want to deduce first order optimality conditions for a (local) minimizer of \eqref{P1}. First we write \eqref{P1} in the equivalent form
\begin{align}\tag{P'}\label{P1'}\left.\begin{aligned}
\min_{u\in H^2\left(\left[0,T\right]\right)^N}\,&\frac{1}{2}\left\|\Psi\left(u\right)-\rho_d\right\|_{L^2\left(\left[0,T\right]\times\R^2\right)}^2+\frac{\beta}{2}\sum_{j=1}^N{c_j\left(\left\|u_j\right\|_{L^2\left(\left[0,T\right]\right)}^2\right.}\\
&\left.+\beta_1\left\|\partial_tu_j\right\|_{L^2\left(\left[0,T\right]\right)}^2+\beta_2\left\|\partial_t^2u_j\right\|_{L^2\left(\left[0,T\right]\right)}^2\right)\\
\text{s.t.}\,& -u_j+1\geq 0,u_j+1\geq 0.
\end{aligned}\right\}\end{align}
Here, the objective function $\overline{\phi}=\overline{\phi}\left(u\right)=\phi\left(\Xi\left(u\right),u\right)$ is a function of only the control.

The constraints will lead to corresponding Lagrange multipliers. In general, to prove their existence, some condition on the constraints is necessary. On this account we verify the constraint qualification of Zowe and Kurcyusz, see \cite{rs}, which is based on a fundamental work of Robinson, \cite{sts}. We rewrite the constraints: $g\left(u\right)\in K$, where $g\left(u\right)=\left(-u+1,u+1\right)\in K$, $K$ denoting the cone of component-wise positive functions in $C\left(\left[0,T\right]\right)^{2N}$. The constraint qualification we have to verify is
\begin{align*}
g'\left(u\right)\left(H^2\left(\left[0,T\right]\right)\right)^N-\left\{k-\lambda g\left(u\right)\mid k\in K,\lambda\geq 0\right\}=C\left(\left[0,T\right]\right)^{2N}.
\end{align*}
In other words, for given $\left(w^+,w^-\right)\in\left(C\left(\left[0,T\right]\right)\right)^{2N}$ we have to find $\delta u\in H^2\left(\left[0,T\right]\right)^N$, $\lambda\in\R_{\geq 0}$, and $k=\left(\theta^+,\theta^-\right)\in\left(C\left(\left[0,T\right]\right)\right)^{2N}$ with $\theta_j^+$, $\theta_j^-\geq 0$, satisfying
\begin{align}\label{findcq}
\left(-\delta u,\delta u\right)-\left(\theta^+,\theta^-\right)+\lambda\left(-u+1,u+1\right)&=\left(w^+,w^-\right).
\end{align}
We abbreviate
\begin{align*}
\vartheta^+:=\max_{\substack{i=1,\dots,N}}{\left\|w_i^+\right\|_\infty},\vartheta^-:=\max_{\substack{i=1,\dots,N}}{\left\|w_i^-\right\|_\infty}.
\end{align*}
Now let
\begin{align*}
\lambda&:=\frac{1}{2}\left(\vartheta^++\vartheta^-\right)+1,\theta_j^+:=\vartheta^+-u_j+1-w_j^+,\theta_j^-:=\vartheta^-+u_j+1-w_j^-,\\
\delta u_j&:=-\frac{1}{2}\left(\vartheta^+\left(u_j+1\right)+\vartheta^-\left(u_j-1\right)\right).
\end{align*}
Obviously, $\lambda\geq 0$ and $\delta u_j$ is of class $H^2$. Furthermore, $\theta_j^+$, $\theta_j^-\in C\left(\left[0,T\right]\right)$ and are $\geq 0$ by choice of $\vartheta^+$, $\vartheta^-$, and feasibility of $u$. Thereby, \eqref{findcq} can easily be verified.

Thus we deduce the following KKT-conditions for a minimizer of \eqref{P1'}. We denote by $M\left(\left[0,T\right]\right)\cong C\left(\left[0,T\right]\right)^*$ the set of regular Borel measures on $\left[0,T\right]$.
\begin{theorem}
Let $\overline{u}$ be a local minimizer of \eqref{P1'}. Then there are Lagrange multipliers $\lambda_j^+$ (corresponding to the constraint $u_j\leq 1$), $\lambda_j^-\in M\left(\left[0,T\right]\right)$ (corresponding to $u_j\geq -1$), $j=1,\dots,N$, satisfying:
\begin{enumerate}[i)]
	\item (Primal feasibility): $\left|\overline{u}_j\right|\leq 1$.
	\item (Dual feasibility): $\lambda_j^+,\lambda_j^-\geq 0$, i.e., $\lambda_j^+v,\lambda_j^-v\geq 0$ for all $v\in C\left(\left[0,T\right]\right)$ with $v\geq 0$.
	\item (Complementary slackness): $\lambda_j^+\left(\overline{u}_j-1\right)=0$, $\lambda_j^-\left(\overline{u}_j+1\right)=0$.
	\item (Stationarity): For all $\delta u\in\left(H^2\left(\left[0,T\right]\right)\right)^N$ it holds that
	\begin{align*}
	&\left<\rho_{\overline{f}}-\rho_d,\rho_{\delta f}\right>_{L^2}\\
	&+\beta\sum_{j=1}^N{c_j\left(\left<\overline{u}_j,\delta u_j\right>_{L^2}+\beta_1\left<\partial_t\overline{u}_j,\partial_t\delta u_j\right>_{L^2}+\beta_2\left<\partial_t^2\overline{u}_j,\partial_t^2\delta u_j\right>_{L^2}\right)}\\
	&=\sum_{j=1}^N{\left(\lambda_j^--\lambda_j^+\right)\delta u_j}
	\end{align*}
	where $\delta f$ is obtained by solving
	\begin{align*}
	\partial_t\delta f+\widehat{p}\cdot\partial_x\delta f+\left(\overline{E}-\widehat{p}^\bot\overline{B}\right)\cdot\partial_p\delta f &= -\left(\delta E-\widehat{p}^\bot\delta B\right)\cdot\partial_p\overline{f},\\
	\partial_t\delta E_1-\partial_{x_2}\delta B &= -j_{\delta f,1}-\delta U_1,\\
	\partial_t\delta E_2+\partial_{x_1}\delta B &= -j_{\delta f,2}-\delta U_2,\\
	\partial_t\delta B+\partial_{x_1}\delta E_2-\partial_{x_2}\delta E_1 &= 0,\\
	\left(\delta f,\delta E,\delta B\right)\left(0\right) &= 0
	\end{align*}
	with $\delta U=\sum_{j=1}^N{\delta u_jz_j}$	and $\left(\overline{f},\overline{E},\overline{B}\right)=\Xi\left(\overline u\right)$.
\end{enumerate}
\end{theorem}

\subsubsection{Adjoint equation}
Considering the optimality conditions above, we note that we have to compute $\overline{\phi}'$ and thus the whole derivative $\Xi'$ at an optimal point $\overline{u}$. However, there is a more efficient way, the adjoint approach, that is to say firstly solve the adjoint equation
\begin{align*}
\partial_yF\left(\Xi\left(u\right),u\right)^*q=-\partial_y\phi\left(\Xi\left(u\right),u\right)
\end{align*}
for the adjoint state $q$ and secondly compute
\begin{align}\label{compderphi}
\overline{\phi}'\left(u\right)=\partial_uF\left(\Xi\left(u\right),u\right)^*q+\partial_u\phi\left(\Xi\left(u\right),u\right).
\end{align}
Here, $y=\left(f,E,B\right)$ denotes the state and $F\left(y,u\right)=0$ the PDE system.

In order to apply these considerations to our problem we have to define $F$ suitably. Here, 'suitably' means that the differentiability of $F$ and the differentiability of the control-to-state operator $\Xi$ have to fit together. In other words, $F\left(y,u\right)$ should be differentiable with respect to the $C\left(0,T;L^2\right)$-norm in the state variable $y=\left(f,E,B\right)$. In the following let
\begin{align*}
M_R&:=\left\{\left(f,E,B\right)\in  C_c^2\left(\left[0,T\right]\times\R^4\right)\times C_c^2\left(\left[0,T\right]\times\R^2;\R^2\right)\times C_c^2\left(\left[0,T\right]\times\R^2\right)\mid\right.\\
&\phan\phan\left.\vphantom{\left\{\left(f,E,B\right)\in  C_c^2\left(\left[0,T\right]\times\R^4\right)\times C_c^2\left(\left[0,T\right]\times\R^2;\R^2\right)\times C_c^2\left(\left[0,T\right]\times\R^2\right)\mid\right.}f\left(t,x,p\right)=0\text{ for all }\left| p\right|\geq R\right\}
\end{align*}
for some $R>0$, and let $M_R$ be equipped with the $C\left(0,T;L^2\right)$-norm. Here, the index '$c$' means 'compactly supported with respect to $x$ and $p$' (or $x$ respectively). Furthermore let
\begin{align*}
Z:=H^1\left(\left[0,T\right]\times\R^4\right)^*\times\left(H^1\left(\left[0,T\right]\times\R^2\right)^*\right)^3\times L^2\left(\R^4\right)^*\times\left(L^2\left(\R^2\right)^*\right)^3.
\end{align*}
Now define $F_R\colon M_R\times \left(H^2\left(\left[0,T\right]\right)\right)^N\rightarrow Z$ via
\begin{align*}
&F_R\left(\left(f,E,B\right),\left(u,\alpha,b\right)\right)\left(g,h_1,h_2,h_3,a_1,a_2,a_3,a_4\right)\\
&=\left(-\int_0^T{\int{\int{\left(\partial_tg+\widehat{p}\cdot\partial_xg+\left(E-\widehat{p}^\bot B\right)\cdot\partial_pg\right)f\,dp}dx}dt}\right.\\
&\phan+\left<g\left(T\right),f\left(T\right)\right>_{L^2}-\left<g\left(0\right),f\left(0\right)\right>_{L^2},\\
&\phan\int_0^T{\int{\left(-E_1\partial_th_1+B\partial_{x_2}h_1+j_{f,1}h_1+U_1h_1\right)dx}dt}\\
&\phan+\left<h_1\left(T\right),E_1\left(T\right)\right>_{L^2}-\left<h_1\left(0\right),E_1\left(0\right)\right>_{L^2},\\
&\phan\int_0^T{\int{\left(-E_2\partial_th_2-B\partial_{x_1}h_2+j_{f,2}h_2+U_2h_2\right)dx}dt}\\
&\phan+\left<h_2\left(T\right),E_2\left(T\right)\right>_{L^2}-\left<h_2\left(0\right),E_2\left(0\right)\right>_{L^2},\\
&\phan\int_0^T{\int{\left(-B\partial_th_3-E_2\partial_{x_1}h_3+E_1\partial_{x_2}h_3\right)dx}dt}\\
&\phan+\left<h_3\left(T\right),B\left(T\right)\right>_{L^2}-\left<h_3\left(0\right),B\left(0\right)\right>_{L^2},\\
&\phan\int{\int{\left(f\left(0\right)-\mathring{f}\right)a_1\,dp}dx},\int{\left(E_1\left(0\right)-\mathring{E_1}\right)a_2\,dx},\int{\left(E_2\left(0\right)-\mathring{E}\right)a_3\,dx},\\
&\phan\left.\int{\left(B\left(0\right)-\mathring{B}\right)a_4\,dx}\vphantom{-\int_0^T{\int{\int{\left(\partial_tg+\widehat{p}\cdot\partial_xg+\left(E-\widehat{p}^\bot B\right)\cdot\partial_pg\right)f\,dp}dx}dt}}\right)
\end{align*}
where $U=\sum_{j=1}^N{u_jz_j}$. After several integrations by parts, it is obvious that $\left(f,E,B\right)$ solves \eqref{CVM} with control $U$ iff $F_R\left(\left(f,E,B\right),u\right)=0$ for any $R>0$ with $\supp_pf\subset B_R$. Since no derivatives of the state $y=\left(f,E,B\right)$ appear above and the state is of class $C_b$, $\partial_yF_R$ exists and is given by
\begin{align*}
&\partial_yF_R\left(\left(f,E,B\right),u\right)\left(\delta f,\delta E,\delta B\right)\left(g,h_1,h_2,h_3,a_1,a_2,a_3,a_4\right)\\
&=\left(-\int_0^T{\int{\int{\left(\left(\partial_tg+\widehat{p}\cdot\partial_xg+\left(E-\widehat{p}^\bot B\right)\cdot\partial_pg\right)\delta f+\left(\delta E-\widehat{p}^\bot \delta B\right)f\cdot\partial_pg\right)dp}dx}dt}\right.\\
&\phan+\left<g\left(T\right),\delta f\left(T\right)\right>_{L^2}-\left<g\left(0\right),\delta f\left(0\right)\right>_{L^2},\\
&\phan\int_0^T{\int{\left(-\delta E_1\partial_th_1+\delta B\partial_{x_2}h_1+j_{\delta f,1}h_1\right)dx}dt}\\
&\phan+\left<h_1\left(T\right),\delta E_1\left(T\right)\right>_{L^2}-\left<h_1\left(0\right),\delta E_1\left(0\right)\right>_{L^2},\\
&\phan\int_0^T{\int{\left(-\delta E_2\partial_th_2-\delta B\partial_{x_1}h_2+j_{\delta f,2}h_2\right)dx}dt}\\
&\phan+\left<h_2\left(T\right),\delta E_2\left(T\right)\right>_{L^2}-\left<h_2\left(0\right),\delta E_2\left(0\right)\right>_{L^2},\\
&\phan\int_0^T{\int{\left(-\delta B\partial_th_3-\delta E_2\partial_{x_1}h_3+\delta E_1\partial_{x_2}h_3\right)dx}dt}\\
&\phan+\left<h_3\left(T\right),\delta B\left(T\right)\right>_{L^2}-\left<h_3\left(0\right),\delta B\left(0\right)\right>_{L^2},\\
&\phan\left.\int{\int{\delta f\left(0\right)a_1\,dp}dx},\int{\delta E_1\left(0\right)a_2\,dx},\int{\delta E_2\left(0\right)a_3\,dx},\int{\delta B\left(0\right)a_4\,dx}\vphantom{-\int_0^T{\int{\int{\left(\left(\partial_tg+\widehat{p}\cdot\partial_xg+\left(E-\widehat{p}^\bot B\right)\cdot\partial_pg\right)\delta f+\left(\delta E-\widehat{p}^\bot \delta B\right)f\cdot\partial_pg\right)dp}dx}dt}}\right)
\end{align*}
for $\left(\delta f,\delta E,\delta B\right)\in M_R$. Note that it is important that $f$ vanishes for $\left| p\right|\geq R$ so that for $i=1,2$ the linear map
\begin{align*}
\left(f,E,B\right)\mapsto\int_0^T{\int{j_{f,i}\cdot dx}dt}\in H^1\left(\left[0,T\right]\times\R^2\right)^*
\end{align*}
is bounded due to
\begin{align*}
\left|\int_0^T{\int{j_{f,i}h_i\,dx}dt}\right|\leq C\left(T,R\right)\left\|f\right\|_{C\left(0,T;L^2\right)}\left\|h_i\right\|_{H^1}
\end{align*}
and hence differentiable.

On the other hand we have
\begin{align*}
\partial_y\phi\left(\left(f,E,B\right),u\right)\left(\delta f,\delta E,\delta B\right)=\left<\rho_f-\rho_d,\rho_{\delta f}\right>_{L^2}.
\end{align*}
Here again, the support condition given in the definition of $M_R$ is important to estimate
\begin{align*}
\left|\int_0^T{\int{\left(\rho_f-\rho_d\right)\rho_{\delta f}\,dx}dt}\right|\leq C\left(T,R\right)\left\|\rho_f-\rho_d\right\|_{L^2}\left\|\delta f\right\|_{C\left(0,T;L^2\right)}
\end{align*}
and
\begin{align*}
\int_0^T{\int{\rho_{\delta f}^2\,dx}dt}\leq C\left(T,R\right)\left\|\delta f\right\|_{C\left(0,T;L^2\right)}^2.
\end{align*}
Now we search for an adjoint state
\begin{align*}
q&=\left(g,h_1,h_2,h_3,a_1,a_2,a_3,a_4\right)\\
&\in Z^*\cong H^1\left(\left[0,T\right]\times\R^4\right)\times\left(H^1\left(\left[0,T\right]\times\R^2\right)\right)^3\times L^2\left(\R^4\right)\times\left(L^2\left(\R^2\right)\right)^3
\end{align*}
satisfying the adjoint system. In other words, after integrating by parts once,
\begin{align*}
&-\int_0^T{\int{\int{\left(\partial_tg+\widehat{p}\cdot\partial_xg+\left(E-\widehat{p}^\bot B\right)\cdot\partial_pg-4\pi\left(\widehat{p}_1h_1+\widehat{p}_2h_2\right)\right)\delta f\,dp}dx}dt}\\
&+\int_0^T{\int{\left(-\partial_th_1+\partial_{x_2}h_3+\int{g\partial_{p_1}f\,dp}\right)\delta E_1\,dx}dt}\\
&+\int_0^T{\int{\left(-\partial_th_2-\partial_{x_1}h_3+\int{g\partial_{p_2}f\,dp}\right)\delta E_2\,dx}dt}\\
&+\int_0^T{\int{\left(-\partial_th_3+\partial_{x_2}h_1-\partial_{x_1}h_2-\int{g\widehat{p}^\bot\cdot\partial_{p}f\,dp}\right)\delta B\,dx}dt}\\
&+\left<g\left(T\right),\delta f\left(T\right)\right>_{L^2}-\left<g\left(0\right)-a_1,\delta f\left(0\right)\right>_{L^2}+\left<h_1\left(T\right),\delta E_1\left(T\right)\right>_{L^2}\\
&-\left<h_1\left(0\right)-a_2,\delta E_1\left(0\right)\right>_{L^2}+\left<h_2\left(T\right),\delta E_2\left(T\right)\right>_{L^2}-\left<h_2\left(0\right)-a_3,\delta E_2\left(0\right)\right>_{L^2}\\
&+\left<h_3\left(T\right),\delta B\left(T\right)\right>_{L^2}-\left<h_3\left(0\right)-a_4,\delta B\left(0\right)\right>_{L^2}\\
&=-\int_0^T{\int{\int{4\pi\left(\rho_f-\rho_d\right)\delta f\,dp}dx}dt}\numb\label{adj}
\end{align*}
for all $\left(\delta f,\delta E,\delta B\right)\in M_R$. Therefore the adjoint state solves the adjoint system
\begin{align}\tag{Ad}\label{adeq}\left.\begin{aligned}
\partial_tg+\widehat{p}\cdot\partial_xg+\left(E-\widehat{p}^\bot B\right)\cdot\partial_pg&=4\pi\left(\widehat{p}_1h_1+\widehat{p}_2h_2\right)+4\pi\left(\rho_f-\rho_d\right),\\
\partial_th_1-\partial_{x_2}h_3&=\int{g\partial_{p_1}f\,dp'},\\
\partial_th_2+\partial_{x_1}h_3&=\int{g\partial_{p_2}f\,dp'},\\
\partial_th_3-\partial_{x_2}h_1+\partial_{x_1}h_2&=-\int{g\widehat{p'}^\bot\cdot\partial_{p}f\,dp'},\\
\left(g,h_1,h_2,h_3\right)\left(T\right)&=0
\end{aligned}\right\}\end{align}
for $\left| p\right|< R$. Since $R>0$ (with $\supp_pf\subset B_R$) is arbitrary, it is natural to demand \eqref{adeq} holds globally on $\left[0,T\right]\times\R^4$. Conversely, if \eqref{adeq} holds for all $p$, then \eqref{adj} holds for for all $\left(\delta f,\delta E,\delta B\right)\in M_R$ for any $R>0$ if we simply set $a_1=g\left(0\right)$, $\left(a_2,a_3,a_4\right)=\left(h_1,h_2,h_3\right)\left(0\right)$. The latter equations are unsubstantial and can be ignored.

In accordance with \eqref{compderphi}, we compute the derivative of $\overline{\phi}$ via
\begin{align*}
\overline{\phi}'\left(u\right)\delta u&=\int_0^T{\int{\left(\delta U_1h_1+\delta U_2h_2\right)dx}dt}+\beta\sum_{j=1}^N{c_j\left(\left<u_j,\delta u_j\right>_{L^2}\right.}\\
&\phan\left.+\beta_1\left<\partial_tu_j,\partial_t\delta u_j\right>_{L^2}+\beta_2\left<\partial_t^2u_j,\partial_t^2\delta u_j\right>_{L^2}\right)
\end{align*}
where $\delta U=\sum_{j=1}^N{\delta u_jz_j}$.

System \eqref{adeq} has to be investigated. It is a final value problem which can easily be turned into an initial value problem via $\widetilde{g}\left(t,x,p\right)=g\left(T-t,-x,-p\right)$ and $\widetilde{h}\left(t,x\right)=h\left(T-t,-x\right)$, so that the left hand sides of the differential equations in \eqref{adeq} do not change. In other words, the hyperbolic system \eqref{adeq} is time reversible.

To show unique solvability of \eqref{adeq}, one can proceed similar to the dealing with \eqref{LVM}. Yet there are some differences, which we will briefly sketch. Firstly, the source terms in the Maxwell equations are not the current densities induced by $g$ but some other moments of $g$. Additionally, even in the fourth equation of \eqref{adeq} a source term appears. Hence we have to prove analogues of Lemmas \ref{field} and \ref{fieldder} with more general source terms. Secondly, the right hand side of the Vlasov equation (and hence a solution $g$) does not have compact support with respect to $p$. But this will not cause any problems since in a representation formula for $h$ there will appear a factor $\partial_pf$ (or first derivatives of $\partial_pf$). Because of the known fact that $f$ is compactly supported with respect to $p$ uniformly in $t$, $x$, we do not have to demand that $g$ has this property. In Section \ref{est} we had to assume this property for the density since the integral defining the current density induced by this density contains the factor $\widehat{p}$ which is obviously not compactly supported in $p$.

\nocite{*}
\bibliography{ocrvm2darxiv}
\bibliographystyle{plain}
\end{document}